\documentclass[a4paper, 11pt]{article}
\usepackage{amsmath,amssymb,esint,amscd,xspace,fancyhdr,color,authblk,srcltx,fontenc,bbm}
\usepackage{hyperref}
\usepackage{hyperref}
\usepackage{cite}

\makeatletter
\newcommand{\leqnomode}{\tagsleft@true}
\newcommand{\reqnomode}{\tagsleft@false}
\makeatother
                     
\newtheorem{theorem}{Theorem}[section]

\newtheorem{definition}[theorem]{Definition}

\newtheorem{lemma}[theorem]{Lemma}

\newtheorem{remark}[theorem]{Remark}

\newtheorem{taggedtheoremx}{Theorem}
\newenvironment{taggedtheorem}[1]
 {\renewcommand\thetaggedtheoremx{#1}\taggedtheoremx}
 {\endtaggedtheoremx}
 
 \newtheorem{taggedingredientx}{Ingredient}
\newenvironment{taggedingre}[1]
 {\renewcommand\thetaggedingredientx{#1}\taggedingredientx}
 {\endtaggedingredientx}

\newenvironment{proof}[1][Proof]{\textbf{#1.} }{\hfill\rule{0.5em}{0.5em}}
{\catcode`\@=11\global\let\AddToReset=\@addtoreset
\AddToReset{equation}{section}

\AddToReset{theorem}{section}

\title{Level-set inequalities on fractional maximal distribution functions and applications to regularity theory}
\author{Thanh-Nhan Nguyen\thanks{Department of Mathematics, Ho Chi Minh City University of Education, Ho Chi Minh City, Vietnam; \texttt{nhannt@hcmue.edu.vn}}, Minh-Phuong Tran\footnote{Corresponding author.}\thanks{Applied Analysis Research Group, Faculty of Mathematics and Statistics, Ton Duc Thang University, Ho Chi Minh City, Vietnam; \texttt{tranminhphuong@tdtu.edu.vn}}}

\date{\today}

\begin{document}
 
\maketitle
\begin{abstract}

The aim of this paper is to establish an abstract theory based on the so-called fractional-maximal distribution functions (FMDs). From the rough ideas introduced in~\cite{AM2007}, we develop and prove some abstract results related to the level-set inequalities and norm-comparisons by using the language of such FMDs. Particularly interesting is the applicability of our approach that has been shown in regularity and Calder\'on-Zygmund type estimates. In this paper, due to our research experience, we will establish global regularity estimates for two types of general quasilinear problems (problems with divergence form and double obstacles), via fractional-maximal operators and FMDs. The range of applications of these abstract results is large. Apart from these two examples of the regularity theory for elliptic equations discussed, it is also promising to indicate further possible applications of our approach for other special topics.

\medskip

\medskip

\medskip

\noindent 

\medskip


\end{abstract}   
                  
\tableofcontents

\section{Introduction}\label{sec:intro}
\textit{1.1. An example - Regularity theory for a class of quasi-linear elliptic equations}

Before embarking on the main objective of this paper, we take into account a class of nonlinear elliptic equations of the type
\begin{align}\label{eq:example}
\mathrm{div}(\mathbb{A}(x,\nabla u)) = \mathrm{div}(|\mathbf{F}|^{p-2}\mathbf{F}), \quad \text{in} \ \ \Omega
\end{align}
as an example and review some recent progresses that have been made in the last several years. As far as we know, researchers have long been interested in regularity theory for linear/nonlinear differential equations and despite plenty of works concerning this kind equations, the picture of regularity for their solutions is somehow incomplete. Let us briefly give a survey of classical and recent regularity results the related to \eqref{eq:example}, in which the main objective is to transfer the regularity of datum $\mathbf{F}$ to the gradient of solutions $\nabla u$ in the norm of some functional space $\mathbb{X}$. The gradient estimates for solutions to equations \eqref{eq:example} can be written as
\begin{align}\label{eq:norm_estX}
\|\nabla u\|_{\mathbb{X}} \le C \|\mathbf{F}\|_{\mathbb{X}},
\end{align}
or in terms of the local Calder\'on-Zygmund estimates
\begin{align}\label{eq:local_czX}
\|\nabla u\|_{\mathbb{X}(B_{R/2})} \le C \|\mathbf{F}\|_{\mathbb{X}(B_{R})} + \texttt{lower-order terms on} \ \nabla u,
\end{align}
where $B_R$ is a open ball in $\mathbb{R}^n$ of radius $R$ such that $B_{R} \subset \Omega$. These types of estimates, in a sharp way, are important in studying many problems concerning nonlinear equations or systems.

Classical regularity results when considering $p$-Laplace equations with $\mathbf{F} \equiv 0$ have been obtained and proved by several authors by years.  In a fundamental work by Ural\'tzeva in~\cite{Ural1968}, it is possible to obtain the $C^{1,\alpha}$ H\"older regularity for $p \ge 2$; and this result was later established for systems by Uhlenbeck~\cite{Uhlenbeck1977}. Afterwards, there has been extensive research on the $p$-harmonic functions since the 1980s such as~\cite{Evans1982, DiBenedetto1983, Lewis1983, Tolksdorf1984} and many recent advances. The very first and well-known regularity result related to the $p$-Laplace equation $\mathrm{div}(|\nabla u|^{p-2}\nabla u) = \mathrm{div}(|\mathbf{F}|^{p-2}\mathbf{F})$, where the estimate~\eqref{eq:norm_estX} established by T. Iwaniec in~\cite{Iwaniec83} when $\mathbb{X}=L^q$, for all $p \le q < \infty$. The method of Iwaniec relies on the use of sharp maximal operators in Harmonic Analysis and priori estimates for solution to homogeneous equations $\mathrm{div}(\mathbb{A}(x_0,\nabla u))=0$. Iwaniec's results were later improved and extended by DiBenedetto and Manfredi in~\cite{DiBenedetto1993} for the systematic elliptic equations when $\mathbb{X}=\text{BMO}$ for $p \ge 2$. During the past several years, there have been extensive studies and promising technical approaches on Calder\'on-Zygmund and regularity for quasi-linear elliptic equations~\eqref{eq:example}. So far many progresses have been made and the literature on regularity theory has been further expanded. In 1993, Caffarelli and Peral in~\cite{CP1998} proposed a very important approach to $L^p$ estimates for quasi-linear elliptic equation $\mathrm{div}(\mathbb{A}(x,\nabla u))=0$. This approach relies on Calder\'on-Zygmund decomposition and the boundedness of Hardy-Littlewood maximal functions. It is valuable to further develop a roadmap for Calder\'on-Zygmund gradient estimates and regularity theory for more general class of nonlinear elliptic/parabolic equations over time. 

In the past couple of years, inspired by this beautiful idea, there have been several attempts to study regularity theory for solutions to elliptic equations in form $\mathrm{div}(\mathbb{A}(x,\nabla u))=\mathrm{div}(|\mathbf{F}|^{p-2}\mathbf{F})$, in which both  local estimates~\eqref{eq:norm_estX} and~\eqref{eq:local_czX} obtained in $\mathbb{X}=L^p$ and $W^{1,p}$ spaces, such as~\cite{Mi3, BCDKS, BW1}. Regarding extendable regularity estimates for solutions to boundary value problems, it enables us to refer~\cite{Wang2, SSB3, BW2, SSB1}- the works by Byun and Wang;~\cite{ByunKim2016, Breit2018, SSB4, BYZ2008, MP12} - for further studies by others, under various weak assumptions on the boundary of the domain. 

Besides, a plenty of interesting approaches yield regularity results. It refers to seminal works by Iwaniec and Sbordone in~\cite{IS1994} with a method of using Hodge decomposition theorem, Lewis in~\cite{Lewis93} with method based on the truncation of certain maximal operators. Or technique in~\cite{CFL1993, KZ}, is particularly useful to study regularity estimates for equations with VMO-coefficients. It is worth mentioning that classical ingredients mainly based on Calder\'on-Zygmund theory, Harmonic Analysis, interpolation inequalities, methods of freezing the coefficients, VMO coefficients via commutator theorem, etc. 

It is worthy to emphasize that Acerbi and Mingione, in an impressive paper~\cite{AM2007}, generated new idea to develop a local Calder\'on-Zygmund theory for degenerate parabolic systems. Therein the authors first proposed a technique with no use of maximal operators, Harmonic analysis free, the basic analysis is an application of DiBenedetto's intrinsic parabolic geometry estimates~\cite{DiBenedettobook} (that is a scaling depending on solution itself) and Vitali's covering lemma. Although only dealing with the quasi-linear parabolic systems, the results for elliptic ones are also well understood (intrinsic cylinders are replaced by balls in the elliptic problems). This new approach opened the majority of intensive works for nonlinear elliptic and parabolic problems, which are still being mined for new and interesting results. There have been a large number of studies conducted, such as the terminology `large-M-inequality' principle; geometrical approach by Byun and Wang in~\cite{SSB4, BW2} adapted to non-smooth domains; types of `good-$\lambda$ bounds' technique by many others~\cite{55QH4, MPT2018, PNJDE, PNCCM, PNnonuniform, PNmix} working with balls instead of cubes. Several regularity results have been extensively treated in more general functional spaces $\mathbb{X}$: $L^p\text{log}L$, Lorentz, Morrey, Lorentz-Morrey spaces, or even Orlicz spaces, studied and addressed in a series of papers~\cite{Baroni2013, Tuoc2018, Phuc2, Phuc2015, Byun2017JDE, FT2018, MPTNsub, Chle2018} with related works.

\textit{1.2. Level-set decay estimates.} Let us briefly describe the idea underlying this effective approach. We refer to the pioneering works in~\cite{AM2007, Mi3, KM2014} for further reading. In order to obtain a local Calder\'on-Zygmund estimates \eqref{eq:local_czX}, in $\mathbb{X}=L^q$ for example, starting with the integral of $\nabla u$, understood in the sense of the Choquet integral as follows:
\begin{align*}
\int{|\nabla u|^q} = q\int_0^\infty{\lambda^{q-1}\left|\{|\nabla u|>\lambda\} \right|d\lambda}, 
\end{align*}
and with change of variable yields
\begin{align*}
\int{|\nabla u|^q} = M^q q \int_0^\infty{\lambda^{q-1}\left|\{|\nabla u|>M\lambda\} \right|d\lambda},
\end{align*}
for every $\lambda>0$ suitably large, where $M>0$ is an arbitrary constant (see for e.g.\cite{Adam1988}). The key point is that, as in \eqref{eq:local_czX}, we want to find a decay estimate for the level-sets $\left|\{|\nabla u|>M\lambda\} \right|$ in terms of level-sets of the datum $\left|\{|\mathbf{F}|>\lambda\} \right|$. As an abstract idea, it states: if the following estimate
\begin{align}\label{eq:levelset}
\left|\{|\nabla u|>M\lambda\} \right| \le M^{-(p+\delta)} \left|\{|\nabla u|>\lambda\} \right| + C_M\left|\{|\mathbf{F}|>\lambda\} \right|,
\end{align} 
holds for some $\delta>0$, then the gradient of solutions $|\nabla u|$ is controlled by the level-sets of data $\mathbf{F}$. More precisely, for a large $M \gg 1$, the $L^q$ regularity estimate of \eqref{eq:example} will be obtained for all $q<p+\delta$. Otherwise speaking, locally we have
\begin{align*}
\int{|\nabla u|^q} \le M^{q-(p+\delta)} \int{|\nabla u|^q}  + C_M\int{|\mathbf{F}|^q}.
\end{align*}
Here, $C_M$ is a positive constant depends only on $M$, and for simplicity, we denote the Lebesgue measure of a set $E \subset \mathbb{R}^n$ by $|E|$ or by $\mathcal{L}^n(E)$ later in our main work. As the reader will see, the proof of level-set inequality \eqref{eq:levelset} is a key step to conclude local Calder\'on-Zygmund type estimates \eqref{eq:local_czX}. Further, it can be seen that the idea expressed here are also valuable to obtain the level sets involving Hardy-Littlewood maximal or fractional maximal operator of $\nabla u$ in terms of the level sets of $\mathbf{F}$, see, for e.g.~\cite{AM2007, Mi3, KM2014} or~\cite{55QH4,MPT2018,PNCRM}, or more literature related to the subject.

\textit{1.3. Motivation and main proposals.}  To the best of the authors’ knowledge, in general, from the example of level-set decay estimate~\eqref{eq:levelset}, it enables us to state: Given two measurable functions $\mathcal{F, G}$, if there holds
\begin{align}\label{eq:levelset_gen}
\left|\{\mathbf{M}_{\alpha}\mathcal{G}> \sigma_{\varepsilon}\lambda\}\right| \le \varepsilon\left|\{\mathbf{M}_{\alpha}\mathcal{G}>\lambda\} \right| + C\left|\{\mathbf{M}_{\alpha}\mathcal{F}>\kappa_{\varepsilon}\lambda\} \right|,
\end{align}
for any $\varepsilon>0$ small enough and $\sigma_{\varepsilon}$, $\kappa_{\varepsilon}>0$, then the gradient estimate~\eqref{eq:norm_estX} can be obtained in terms of $\mathbf{M}_{\alpha}$ as
\begin{align*}
\|\mathbf{M}_{\alpha}\mathcal{G}\|_\mathbb{X} \le C\|\mathbf{M}_{\alpha}\mathcal{F}\|_\mathbb{X}.
\end{align*}

A question that arises pretty naturally here concerning some sufficient conditions for the level-set inequality~\eqref{eq:levelset} or likewise to be valid. What are the main tools behind the proof of~\eqref{eq:levelset_gen} or how it turns out the idea to construct conditions for $\mathcal{F}$ and $\mathcal{G}$ to sharply achieve~\eqref{eq:levelset_gen}-type inequality? The primary goal of this paper is to answer these questions. With this as motivation, in this study we discuss on some key ingredients for the proof of such type of level-set estimates. 

Inspired by the ideas coming from aforementioned example, if one can find two additive functions $\varphi$ and $\psi$ such that: $\varphi$ belongs to a so-called reverse H{\"o}lder class; and function $\psi$ is able to be controlled by $\varepsilon\mathcal{G}$, for all $\varepsilon>0$, via a local integral estimate (see~\eqref{eq:ReH} and ingredient (A2) in Section~\ref{sec:ingredient} below). To better understand these key ingredients, let us turn back to the abstract theory of nonlinear elliptic equations~\eqref{eq:example}. Here, a version of Gehring's lemma is applied to improve the degree of gradient integrability of weak solutions $v$ to homogeneous equations of type
\begin{align}\label{eq:refeq}
\mathrm{div}(\mathbb{A}(x,\nabla v)) =0, \  \text{in} \ \mathcal{B} \quad \text{and} \quad v=u, \ \text{on} \ \partial\mathcal{B},
\end{align}
whenever $\mathcal{B}$ is an open ball in $\Omega$, see~\cite{Gehring} and later many different versions have been established (see, for e.g., \cite[Theorem 6.7]{Giu},\cite{Iwaniec1995}). As a result of Gehring's lemma, the self-improving property of a well-known inequality, called reverse H\"older integral inequality with increasing supports: if $v$ is the unique solution to reference problem~\eqref{eq:refeq}, then there exists a number $\gamma>1$ depending on $n$, $p$ and the structure of operator $\mathbb{A}$ such that the following inequality holds
\begin{align*}
\left(\fint_{B_\rho}|\nabla v|^{\gamma p} dx\right)^{\frac{1}{\gamma p}}\leq C\left(\fint_{B_{2\rho}}|\nabla v|^p dx\right)^{\frac{1}{p}},
\end{align*}
for all $B_{2\rho} \subset \mathcal{B}$. As we shall see, the function $\varphi$ here plays a role of $\nabla v$. On the other hand, function $\psi$ is in fact the difference between gradients of solutions to equations~\eqref{eq:example} and~\eqref{eq:refeq}, also known as the comparison estimates, must be established in most of research papers. In the context of regularity estimates above-described, these technical ingredients are helpful to recover integrability information of solutions from data, as in~\eqref{eq:levelset}. 

\textit{1.4. Highlights and Significance.}  In accordance with the questions arising before, the discussion leads us to another interesting tool for abstract results. From another point of view, level-set inequality \eqref{eq:levelset_gen} might actually work on the idea of \emph{fractional-maximal distribution functions} (FMD), more precisely as
\begin{align}
\label{eq:levelset_dist}
d_\mathcal{G}^{\alpha}(\mathcal{B},\sigma_{\varepsilon}\lambda) \le \varepsilon d_\mathcal{G}^{\alpha}(\mathcal{B},\lambda) + Cd_\mathcal{F}^{\alpha}(\mathcal{B},\kappa_{\varepsilon}\lambda).
\end{align}
The construction of such appropriate tool can provide new insights of the technical approach when introducing or discussing on regularity theory and its applications. For the sake of clarity and completeness, this level-set type \eqref{eq:levelset_dist} will be explained in Section~\ref{sec:levelset}.

The aim of this paper is two-fold. First we discuss the basic ingredients to formulate the level-set estimates in terms of fractional maximal functions. Specifically, by using the language of such \emph{fractional-maximal distribution functions}, we provide a newer landmark for the `good-$\lambda$' type bounds technique, that has important theoretical implications in regularity and Calder\'on-Zygmund type estimates. On a different direction, researching the regularity theory of nonlinear elliptic equations, that is also linked to the double obstacle problems become a new trend in nonlinear PDEs. Secondly, as an application of the abstract setting for this technique, we shall prove some global regularity estimates for nonlinear elliptic problems. In particular, there are two separate issues discussed here. On the one hand, we develop the level-set decay estimate~\eqref{eq:levelset} (in terms of fractional maximal operators $\mathbf{M}_{\alpha}$) to establish the global regularity estimates for a wide class of nonhomogeneous quasi-linear elliptic equations as follows
\begin{align}\tag{$\mathbf{P}$}
\label{eq:app_intro}
\mathrm{div}(\mathbb{A}(x,\nabla u))  = \ \mathrm{div}(\mathbb{B}(x,\mathbf{F}))\ \text{in} \ \Omega, \quad u =\mathsf{g}\ \text{on}\  \partial \Omega,
\end{align}
where $\mathbf{F} \in L^p(\Omega;\mathbb{R}^n)$ with boundary data $\mathsf{g} \in W^{1,p}(\Omega)$ for $p \in (1,n]$. This form of equations appears naturally in many engineering or science problems. Here, we focus our attention on the appearance of the degeneracy parameter $\varsigma \in [0,1]$ in the standard assumptions of $\mathbb{A}$ (growth and ellipticity conditions, see Section~\ref{sec:app}). Apart from the regularity results described in example above, there have been  remarkable contributions pertaining to regularity theory for degenerate problems with $\varsigma=0$, see~\cite{Phuc2015,Baroni2013,SSB3,55QH4,Mi3, KM2014, Duzamin2,55DuzaMing,PNmix,MPT2018} and many extensive literature so far. In this journey, we confine ourselves with regularity estimates for $\varsigma \ge 0$. On the other hand, as the second application, we want to apply the proposed technique to \emph{nonlinear elliptic double obstacle problems}, where the solutions are constrained to lie between two fixed obstacle functions: $f_1 \le u \le f_2$ (see Section~\ref{sec:app} below, for detailed description). This constrained variational problem is an interesting topic that has a wide range of applications in elasto-plasticity, mathematical finance, optimal control problem, groundwater hydrology, the study of a soap film, equilibrium of an elastic membrane, transactions costs and other sciences (see reference books in~\cite{Friedman,Troianiello,KS1980, Rodfrigues1987} for further mathematical problems and applications).  Significant progress has been made for one-sided obstacle problems in~\cite{MZ1986,Choe1991, Eleuteri2007, EH2011, BDM2011, BS2012, BCW2012, EH2008,Caffarelli_obs,BPR2013} and many references given therein. However, there seems not too much works on the double obstacle case, even though it also arises in many applications. In this paper, along with the works~\cite{MMV1989, KZ1991,Lieberman1991,Choe2016,RT2011}, somewhat extends the results in~\cite{BR2020}, we prove the global gradient estimates of solutions to double obstacle problems by our technical argument, via the theory of FMD. 

One of the new aspects of our work is that we deal with fractional maximal operators. Together with Hardy-Littlewood maximal operators, this is one of important variants in analysis and PDEs to study differentiability properties of functions. Fractional maximal operator, usually denoted by $\mathbf{M}_{\alpha}$, whose definition will be essentially given in the next section, is useful tool to acquire gradient estimates for solutions for a large class of quasi-linear elliptic/parabolic equations (see~\cite{Duzamin2, 55DuzaMing,KM2012, KM2014} and many many research papers so far). In this study, we employ $\mathbf{M}_{\alpha}$ to take advantage of the efficiency of the proposed technique. To be more precise, gradient estimates of solutions to the general problems~\eqref{eq:app_intro} are preserved under fractional maximal operators.

Why fractional maximal operators come into play? - In~\cite{Adam1975}, $\mathbf{M}_{\alpha}$ has a connection to the Riesz potential $\mathbf{I}_\alpha$ (fractional integral operator) in the following point-wise inequality:
\begin{align}\label{eq:MI}
\mathbf{M}_{\alpha} f(x) \lesssim \mathbf{I}_{\alpha} f(x), \quad \text{for every} \ x \in \mathbb{R}^n,
\end{align}
and additionally, the fractional maximal function $\mathbf{M}_\alpha f$ and Riesz potential $\mathbf{I}_\alpha f$ are often comparable in norm,~\cite{MW}. It is observed that fractional maximal operator and the Riesz potential are connected via relation~\eqref{eq:MI}, allowing both size and oscillations of solutions and their derivatives, including `fractional derivatives' $\partial^{\alpha} u$ to be controlled (see~\cite{KM2014}). Henceforth, it enables us to exploit the $\mathbf{M}_{\alpha}$ to transfer the level-set information from given data $\mathbf{F}$ to $\nabla u$. 

One more to emphasize in this study, global regularity results in Section~\ref{sec:app} will be obtained in the setting of Lorentz and Orlicz spaces, respectively. Moreover, an extra attempt has also been made to study results in the Orlicz-Lorentz setting (a generalization of Orlicz and Lorentz spaces - see~\cite{MontOL, Kaminska}) in Section~\ref{sec:app} below. As a natural way analogous to this work, we plan to combine these results and proposed technique in order to study the boundary value problems for nonlinear parabolic equations/systems, in a forthcoming paper.

Going far beyond, this work has significance for its genre. By rephrasing idea from the references already given (including the relevant contributions and ours), in view of the generalizations to `good-$\lambda$' level-set inequalities, this paper develops a general and robust approach to build on the higher regularity via the use of FMD. Apart from being an interesting in its own, this work reveals a wider perspective of such technique in modern analysis. This paper gives a flavor to the reader of the essence behind the proof of Calder\'on-Zygmund-type estimates, which attracts a number of interesting works during last decades. We call the attention of the reader for enlightening paper~\cite{AM2007} and further papers related to this approach. Being a contribution to the study of regularity theory for nonlinear elliptic/parabolic problems, we believe that this paper can provide an inviting reading on the topic, especially to newcomers. 

\textit{1.5. Main results and Outline of the paper.}  
Let us now state our main results which will be summarized into two following theorems. In theorem \ref{theo-A}, we discuss some sufficient conditions for the validity of FMD inequalities. In general, these conditions can be represented by the key ingredients in our statements (see Section \ref{sec:ingredient} for details). Next, arising from what obtained in Theorem \ref{theo-A}, Theorem \ref{theo-B} enables us to obtain the norm-comparisons in the setting of several spaces, such as: Lorentz spaces, Orlicz spaces and Orlicz-Lorentz spaces. However, it is a remarkable fact that the proofs of Theorems \ref{theo-A} and \ref{theo-B} above are splitted into separate parts, to be convenient to the readers. These new abstract results in this paper allow the application of any type of regularity theory (Calder\'on-Zygmund estimates) for partial differential equations. As already said, we here only deal with two applications: for a general nonhomogeneous quasilinear elliptic equations and for quasilinear elliptic double obstacle problems.
\begin{taggedtheorem}{A}\label{theo-A}
Let $\gamma>1$ and two functions $\mathcal{F}$, $\mathcal{G} \in L^1(\Omega; \mathbb{R}^+)$ satisfy the global comparison in~\ref{ing:A3}.
\begin{itemize}
\item[i)] If $\mathcal{F}$, $\mathcal{G}$ satisfy the local comparison~\ref{ing:A2_1} then for every $\alpha \in [0,\frac{n}{\gamma})$ there exists $\varepsilon_0 \in (0,1)$ such that the following fractional-maximal distribution inequality 
\begin{align}\label{ineq:dG-0}
d_{\mathcal{G}}^{\alpha}(\Omega; \sigma \lambda)   \le C \varepsilon d_{\mathcal{G}}^{\alpha}(\Omega; \lambda)  + d_{\mathcal{F}}^{\alpha}(\Omega; \kappa \lambda),
\end{align}
holds for all $\lambda>0$ and $\varepsilon \in (0,\varepsilon_0)$,  with $\sigma = \varepsilon^{-\frac{n-\alpha \gamma}{n\gamma}}$ and $\kappa = \varepsilon c_{\varepsilon}^{-1}$.
\item[ii)] If $\mathcal{F}$, $\mathcal{G}$ satisfy the local comparison~\ref{ing:A2_2} then there exists $\sigma_0 = \sigma_0(n,\tilde{c})>0$ such that the fractional-maximal distribution inequality~\eqref{ineq:dG-0} holds for all $\lambda>0$ and $\varepsilon \in (0,1)$ and for some $\kappa  \in (0,\varepsilon)$. 
\end{itemize}
\end{taggedtheorem}

\begin{taggedtheorem}{B}\label{theo-B}
Let $\gamma>1$ and two functions $\mathcal{F}$, $\mathcal{G} \in L^1(\Omega; \mathbb{R}^+)$ satisfy the global one~\ref{ing:A3}.
\begin{itemize}
\item[i)] If $\mathcal{F}$, $\mathcal{G}$ satisfy the local comparison~\ref{ing:A2_1} then for every $\alpha \in [0,\frac{n}{\gamma})$, $0< q < \frac{n\gamma}{n-\alpha\gamma}$ and $0<s \le \infty$ there exists a constant $C>0$ such that
\begin{align}\label{eq:main-B-0}
\|\mathbf{M}_{\alpha}\mathcal{G}\|_{L^{q,s}(\Omega)} & \le  C  \|\mathbf{M}_{\alpha}\mathcal{F}\|_{L^{q,s}(\Omega)}.
\end{align}
Moreover, given Young function $\Phi \in \Delta_2$ then there exists $\tilde{q}>0$ such that the following estimate
\begin{align}\label{eq:L-O-0}
\|\mathbf{M}_{\alpha}\mathcal{G}\|_{L^{\Phi}(q,s)(\Omega)} \le C \|\mathbf{M}_{\alpha}\mathcal{F}\|_{L^{\Phi}(q,s)(\Omega)},
\end{align}
holds for every $\alpha \in [0,\frac{n}{\gamma})$,  $0 < q < \tilde{q}$ and $0<s \le \infty$.
\item[ii)] If $\mathcal{F}$, $\mathcal{G}$ satisfy the local comparison~\ref{ing:A2_2} then both inequalities~\eqref{eq:main-B-0} and~\eqref{eq:L-O-0} even hold for all $\alpha \in [0,n)$, $0<q<\infty$ and $0 < s \le \infty$. 
\end{itemize}
\end{taggedtheorem}

We conclude the introductory section by outlining the content of this paper. In Section~\ref{sec:ingredient}, we introduce some general notation and basic definitions that will be used throughout the paper. Furthermore, this section is also dedicated to discuss on some crucial ingredients emerged in our approach. Section~\ref{sec:levelset} will establish level-set inequalities by specifying via FMDs; then the definitions of considered functional spaces can be reformulated in terms of such distribution functions.  We also state our chief result in Section~\ref{sec:levelset}. The next section~\ref{sec:abstract} brings these FMD inequalities back to the norm inequalities. We also state and prove some abstract results related to comparisons for different functional spaces (Lorentz, Orlicz and the Orlicz-Lorentz spaces). At the end, we will present two applications where our results take place. For a wider understanding, Section~\ref{sec:app} will detail the global gradient estimates for a general class of quasi-linear elliptic equations via fractional maximal operators based on the idea of FMD established in Section~\ref{sec:levelset} and~\ref{sec:abstract}; and further the global regularity results are also driven with elliptic double obstacle problems, thus providing the complete picture for its applications.

\section{Main ingredients}\label{sec:ingredient}
\subsection{Notation and definitions}\label{sec:notation_def}
In this section, to be convenient for the readers, we first go over some notation and preliminary definitions that will be frequently used in the rest of the paper.

\begin{itemize}
\item Throughout the paper, we employ the letter $C$ to denote the universal positive constant that might be different from line to line.  Furthermore, all constants starting by $C$, such as $C, C_i$ for example, are assumed to be larger than or equal to one and the dependencies on prescribed parameters will be emphasized between parentheses.
\item The domain $\Omega$ is assumed to be an open bounded subset of $\mathbb{R}^n$, for $n \ge 2$. 
\item As apparent from introductory section, we use the denotation $\mathcal{L}^n(E)$ or some time $|E|$ with simplicity, for the Lebesgue measure of a set $E$ in $\mathbb{R}^n$. 
\item In what follows, for a measurable map $h \in L^1_{\mathrm{loc}}(\mathbb{R}^n)$ over subset $E$ of $\mathbb{R}^n$, we shall denote 
\begin{align*}
\fint_E{h(x)dx} = \frac{1}{\mathcal{L}^n(E)}\int_E{h(x)dx},
\end{align*}
as its mean value integral.
\item The open ball in $\mathbb{R}^n$ with center $x$ and radius $\rho>0$ is the set $\{y \in \mathbb{R}^n: |y-x|<\rho\}$, will be abbreviated by $B_\rho(x)$, as usual, for every $x \in \Omega$. In the context, when the center $x$ lies on $\partial\Omega$, we also denote $\Omega_{\rho}(x) := B_{\rho}(x) \cap \Omega$, described as the ``\textit{surface ball}'' in $\mathbb{R}^n$. 
\item For the sake of convenience, by an abuse of notation, as in level-set example \eqref{eq:levelset} above-mentioned and in what follows, the set $\{x \in \Omega: |g(x)| > \Lambda\}$ is also written as $\{|g|>\Lambda\}$.
\end{itemize}

Let us now pass to the definitions of Hardy-Littlewood maximal, fractional maximal functions and the Riesz potential in the spirit of~\cite{K1997, KS2003}. To our knowledge, these significant operators are the most useful tools providing the understanding in Harmonic Analysis, partial differential equations and nonlinear potential theory. And as we shall see, these operators also play a crucial role in our discussion here.

\begin{definition}[Fractional maximal function]\label{def:Malpha}
Let $0 \le {\alpha} \le n$ and $f \in L^1_{\mathrm{loc}}(\mathbb{R}^n)$. Then, the fractional  maximal function $\mathbf{M}_{\alpha} f$ of $f$ is defined by
\begin{align}\label{eq:Malpha}
\mathbf{M}_{\alpha} f(x) = \sup_{\varrho>0}{\varrho^{\alpha} \fint_{B_\varrho(x)}{|f(y)|dy}}, \quad x \in \mathbb{R}^n.
\end{align}
\end{definition}
It is clear to see that when ${\alpha}=0$, $\mathbf{M}_0 \equiv \mathbf{M}$ is the classical Hardy-Littlewood maximal function, and we drop the subscript ${\alpha}$ in this case.

\begin{definition}[Cutoff fractional maximal functions]\label{def:cut-off}
Let $0\le \alpha \le n$ and $f \in L^1_{\mathrm{loc}}(\mathbb{R}^n)$. We define two cutoff fractional maximal functions of $f$ corresponding to $\mathbf{M}_{\alpha}f$ in~\eqref{eq:Malpha} at level $r>0$ as
\begin{align}\label{eq:frac-Malpha}
{\mathbf{M}}^{r}_{\alpha}f(x)  &= \sup_{0<\varrho<r} \varrho^{\alpha} \fint_{B_\rho(x)}f(y)dy; \ \ {\mathbf{T}}^{r}_{\alpha}f(x) = \sup_{\varrho \ge r} \varrho^{\alpha}\fint_{B_\varrho(x)}f(y)dy.
\end{align}
\end{definition}

\begin{definition}[Riesz potential]\label{def:Riesz}
Given $n \ge 2$ and ${\alpha} \in (0,n)$, the fractional integral operator or Riesz potential $\mathbf{I}_{\alpha} f$ of a measurable function $f\in L^1_{\mathrm{loc}}(\mathbb{R}^n;\mathbb{R}^+)$ is defined as the convolution
\begin{align*} 
\mathbf{I}_{\alpha} (f)(x) \equiv (\mathbf{I}_{\alpha} * f)(x) = \int_{\mathbb{R}^n}{\frac{f(y)}{|x-y|^{n-{\alpha}}}dy}, \quad x \in \mathbb{R}^n.
\end{align*}
\end{definition}
From definitions of $\mathbf{M}_{\alpha}$ and $\mathbf{I}_{\alpha}$, as already shown in \eqref{eq:MI}, particularly there holds
\begin{align*}
2^{{\alpha}-n}\mathcal{L}^n(B_1(0))^{\frac{n-{\alpha}}{n}}\mathbf{M}_{\alpha}(f)(x) \le \mathbf{I}_{\alpha}(f)(x), \quad \text{for every} \ \ x \in \mathbb{R}^n,
\end{align*}
for any non-negative measurable function $f$ on $\mathbb{R}^n$, see~\cite{Mi3}. The reader is referred to the textbooks by Stein~\cite{Stein} or Grafakos~\cite{55Gra} for the basic properties of these operators.

In the sequel, let us take some definitions regarding the technical heart of this paper. 
\begin{definition}[Quasi-triangle class]
\label{asmp:QT}
Let $B \subset \Omega$, we denote by $\mathrm{Q}(B)$ a quasi-triangle class of all triplets of measurable functions $(\mathcal{G}, \varphi,\psi)$ defined in $B$ if there exists a constant $\tilde{c} \ge 1$ such that
\begin{align}\label{cond:phi-psi}
\mathcal{G} \le \tilde{c} (\varphi + \psi), \quad \varphi \le \tilde{c} (\mathcal{G} + \psi), \quad \psi \le \tilde{c}(\mathcal{G} + \varphi),  \qquad \mbox{ in } \ B.
\end{align}
\end{definition}
\begin{remark}
As an example, let us consider two measurable functions $u, v$ defined in $\Omega$. It is very easy to check that the triplet $\left(|u|^p, |v|^p, |u-v|^p\right)$ satisfies the quasi-triangle inequality~\eqref{cond:phi-psi} in $\Omega$ with $\tilde{c} = 2^{p-1}$.  
\end{remark}
\begin{definition}[Reverse H\"older class]
\label{def:ReH}
Let $\gamma >1$ and $\varphi \in L^1(\Omega_{2r}(\nu))$ for $r>0$ and $\nu \in \mathbb{R}^n$. We say that the function $\varphi$ belongs to the reverse H{\"o}lder class $\mathcal{RH}^{\gamma}(\Omega_{r}(\nu))$ if there exists a constant $C = C(n,\gamma)>0$ such that
\begin{align}\label{eq:ReH}
\left(\fint_{\Omega_{r}(\nu)}{\left[ \varphi (x)\right]^\gamma dx} \right)^\frac{1}{\gamma} \le C\fint_{\Omega_{2r}(\nu)} \varphi (x) dx.
\end{align} 
\end{definition}

\subsection{Key ingredients}
\label{sec:key}
\begin{taggedingre}{$\mathbf{(A1)}$}
\label{ing:A1}
For given $r>0$ and $\nu \in \mathbb{R}^n$, $\varphi \in \mathcal{RH}^{\gamma}(\Omega_{r}(\nu))$.\end{taggedingre}
\begin{taggedingre}{$\mathbf{(A2_1)}$}[Local comparison]
\label{ing:A2_1}
Let us fix $r_0>0$, we say that $\mathcal{F}$, $\mathcal{G}$ satisfy local comparison $(A2)_1$ if: for every $\nu \in \overline{\Omega}$ and $r \in (0, r_0/2]$, one can find two measurable functions $\varphi$, $\psi$ defined in $\Omega_{2r}(\nu)$ such that $(\mathcal{G}, \varphi,\psi) \in \mathrm{Q}(B_{2r}(\nu))$ with constant $\tilde{c}>0$, $\varphi \in \mathcal{RH}^{\gamma}(\Omega_{r}(\nu))$ and the following estimate
\begin{align}\label{eq:LC}
\fint_{B_{r}(\nu)}{\psi(x) dx} & \le \varepsilon   \fint_{B_{2r}(\nu)}{\mathcal{G}(x) dx}  + c_{\varepsilon}   \fint_{B_{2r}(\nu)}{\mathcal{F}(x)dx},
\end{align}
holds for all $\varepsilon \in (0,1)$.
\end{taggedingre}
\begin{taggedingre}{$\mathbf{(A2_2)}$}[Local comparison]
\label{ing:A2_2}
Let us fix $r_0>0$, we say that $\mathcal{F}$, $\mathcal{G}$ satisfy local comparison $(A2)_2$ if: for every $\nu \in \overline{\Omega}$ and $r \in (0, r_0/2]$, one can find two measurable functions $\varphi$, $\psi$ defined in $\Omega_{r}(\nu)$  such that $(\mathcal{G}, \varphi,\psi) \in \mathrm{Q}(B_{r}(\nu))$  with constant $\tilde{c}>0$, inequality~\eqref{eq:LC} holds for all $\epsilon \in (0,1)$ and 
\begin{align}
\label{eq:Linf}
\|\varphi\|_{L^{\infty}(B_{r}(\nu))} \le C \fint_{B_{2r}(\nu)} \left(\mathcal{G}(x) + \mathcal{F}(x)\right)dx.
\end{align}
\end{taggedingre}
\begin{remark}
\label{rem:ingreA2}
It is important to underline here that ingredient~\ref{ing:A2_1} and~\ref{ing:A2_2} are independently utilized, allowing to derive two separate results in this paper.
\end{remark}
\begin{taggedingre}{$\mathbf{(A3)}$}[Global comparison]
\label{ing:A3}
We say that $\mathcal{F}$, $\mathcal{G}$ satisfy ingredient (A3) if there exists a positive constant $C$ such that
\begin{align}\label{eq:GC}
\fint_{\Omega}{\mathcal{F}(x) dx} \le C \fint_\Omega{\mathcal{G}(x)dx}.
\end{align}
\end{taggedingre}
\begin{taggedingre}{$\mathbf{(A4)}$}[Covering lemma]
\label{ing:A4}
The substitution of Calder\'on-Zygmund-Krylov-Safonov decomposition leading to the following important key lemma, that is a standard result in measure theory.
\begin{lemma}\label{lem:Cover}
Consider two measurable subsets $\mathcal{P}\subset \mathcal{Q}$ of $\Omega$. 
Assume that there are two constants $\varepsilon \in (0,1)$ and  $r \in \left(0,r_0\right]$ such that
\begin{itemize}
\item[i)] $\mathcal{L}^n\left(\mathcal{P}\right) \le \varepsilon \mathcal{L}^n\left(B_{r}(0)\right)$;
\item[ii)] for all $\xi \in \Omega$ and $\varrho \in (0,r]$, if $\mathcal{L}^n\left(\mathcal{P} \cap B_{\varrho}(\xi)\right) > \varepsilon \mathcal{L}^n\left(B_{\varrho}(\xi)\right)$ then $\Omega_{\varrho}(\xi)  \subset \mathcal{Q}$. 
\end{itemize}
Then there exists a constant $C=C(n)>0$ such that $\mathcal{L}^n\left(\mathcal{P}\right)\leq C \varepsilon \mathcal{L}^n\left(\mathcal{Q}\right)$.
\end{lemma}
\end{taggedingre}
This lemma is a version of Calder\'on-Zygmund (or Vitali type) covering lemma that allows us to work with balls instead of cubes, see~\cite[Lemma 4.2]{CC1995} or~\cite{Vitali08}.  As far as we know, there have been various modified versions applications related to this  famous covering lemma. For instance, a parabolic version developed by N. V. Krylov and M. V. Safonov in~\cite{KS1980}, or some interesting works by others in~\cite{CP1998, Wang2}, etc. Lemma~\ref{lem:Cover} is one of the key roles to measure estimates concluded in our proofs.  For further reading on such decomposition lemma, we strongly recommend the reader to~\cite{CC1995, CP1998, Wang2} and the references therein.

\section{Fractional-maximal distribution inequalities}\label{sec:levelset}
In the context of our work, the purpose of this section is to recall the distribution functions introduced in~\cite{55Gra, AH}, that will be the basic tool to construct the definition of Lorentz spaces. Next, this part is to familiarize the reader with \emph{fractional-maximal distributions} (FMD) and some properties concerning this type of distribution function. Also, in the language of fractional-maximal distribution, we point out the boundedness properties of fractional maximal operators. Furthermore, in this section, we are going to state and prove the key results of the paper, Theorems~\ref{theo:G-lam-A} and~\ref{theo:G-lam-B}, being the idea of `good-$\lambda$' technique and in the spirit of the FMD.

\subsection{Fractional-maximal distribution functions (FMD)}
\label{sec:FMD}
In what follows, we always assume $\Omega$ is an arbitrary open domain in $\mathbb{R}^n$, $n \ge 2$. 
\begin{definition}[Distribution function,~\cite{55Gra, AH}]\label{def:df}
The distribution function of a Lebesgue measurable function $f$ on $\Omega$ is the function $d_f$ defined in $[0,\infty)$ as follows
\begin{align*}
d_f(\Omega;\lambda) = \mathcal{L}^n \left(\left\{x \in \Omega: \ |f(x)| > \lambda\right\}\right), \quad \lambda \ge 0.
\end{align*}
\end{definition}
\begin{definition}[Fractional-maximal distribution function (FMD)] \label{def:dG}
Let $0 \le \alpha \le n$ and $\mathcal{G} \in L^{1}_{\mathrm{loc}}(\mathbb{R}^n)$. Then, the fractional-maximal distribution function of $\mathcal{G}$, denoted by $d_{\mathcal{G}}^{\alpha}$, is the distribution function of $\mathbf{M}_{\alpha} \mathcal{G}$. More precisely, for every $\lambda \ge 0$, we define
\begin{align}\label{eq:def-dG}
d_{\mathcal{G}}^{\alpha}(\Omega; \lambda) := d_{\mathbf{M}_{\alpha}\mathcal{G}}(\Omega; \lambda) = \mathcal{L}^n \left(\mathcal{V}_{\alpha}(\mathcal{G}; \lambda) \cap \Omega\right), 
\end{align}
where the measurable subset $\mathcal{V}_{\alpha}(\mathcal{G}; \lambda)$ of $\mathbb{R}^n$ is defined by
\begin{align*}
\mathcal{V}_{\alpha}(\mathcal{G}; \lambda) := \left\{x \in \mathbb{R}^n: \ \mathbf{M}_{\alpha} \mathcal{G}(x)> \lambda\right\}. 
\end{align*}
\end{definition}
We shall denote by $\mathcal{V}_{\alpha}^{c}(\mathcal{G}; \lambda)$ the complement of  $\mathcal{V}_{\alpha}(\mathcal{G}; \lambda)$ in $\mathbb{R}^n$, this means 
\begin{align*}
\mathcal{V}_{\alpha}^{c}(\mathcal{G}; \lambda) = \left\{x \in \mathbb{R}^n: \ \mathbf{M}_{\alpha} \mathcal{G}(x) \le \lambda\right\}.
\end{align*}

Similar to what distribution functions give, one can observe that the FMD depends only on the fractional maximal operator  $\mathbf{M}_\alpha f$ and it provides information about the size (Lebesgue norm information) of this operator. In connection with boundedness property of fractional maximal functions, we will now discuss the following important properties of $d_\mathcal{G}^\alpha$, in Lemma~\ref{lem:M_alpha} below. 

\begin{lemma}\label{lem:M_alpha}
For every $s \ge 1$ and $\alpha \in \left[0,\frac{n}{s}\right)$, there exists $C=C(n,\alpha,s)>0$ such that
\begin{align}\label{eq:M-alpha}
d_{\mathcal{G}}^{\alpha}(\mathbb{R}^n; \lambda)  \le C \left({\lambda}^{-1}{\|\mathcal{G}\|_{L^s(\mathbb{R}^n)}}\right)^{\frac{ns}{n-\alpha s}},
\end{align}
for any $\lambda>0$ and $\mathcal{G} \in L^s(\mathbb{R}^n)$.
\end{lemma}
\begin{proof}
For any $x \in \mathbb{R}^n$, the definition of $\mathbf{M}_{\alpha}$ in~\eqref{def:Malpha} and H{\"o}lder's inequality give us
\begin{align*}
\left[\mathbf{M}_{\alpha} \mathcal{G}(x)\right]^s &= \left(\sup_{\varrho>0} \varrho^{\alpha} \fint_{B_{\varrho}(x)}|\mathcal{G}(y)|dy\right)^s  \le \sup_{\varrho>0} \varrho^{\alpha s} \fint_{B_{\varrho}(x)}|\mathcal{G}(y)|^sdy = \mathbf{M}_{\alpha s} (\mathcal{G}^s)(x).
\end{align*}
Moreover, let us denote $\lambda_0 = \|\mathcal{G}\|^{s}_{L^s(\mathbb{R}^n)}$, one has
\begin{align*}
\mathbf{M}_{\alpha s} (\mathcal{G})^s(x) & = \sup_{\varrho>0} \left(\varrho^{-n} \int_{B_{\varrho}(x)}|\mathcal{G}(y)|^s dy\right)^{\frac{n-\alpha s}{n}} \left(\int_{B_{\varrho}(x)}|\mathcal{G}(y)|^s dy\right)^{\frac{\alpha s}{n}}  \le C \left[\mathbf{M}\mathcal{G}^s(x)\right]^{1 - \frac{\alpha s}{n}} \lambda_0^{\frac{\alpha s}{n}}.
\end{align*}
Combining two above inequalities and definition of $d_{\mathcal{G}}^{\alpha}$ in~\eqref{eq:def-dG}, for every $\lambda>0$ there holds
\begin{align*}
d_{\mathcal{G}}^{\alpha}(\mathbb{R}^n; \lambda)   & \le  \mathcal{L}^n \left( \left\{ \mathbf{M}\mathcal{G}^s > C\lambda_0^{-\frac{\alpha s}{n - \alpha s}} \lambda^{\frac{n s}{n-\alpha s}}\right\} \right)  \le  C \lambda_0^{\frac{\alpha s}{n - \alpha s}} \lambda^{-\frac{n s}{n-\alpha s}} \int_{\mathbb{R}^n}|\mathcal{G}(x)|^s dx,
\end{align*}
which allows us to conclude~\eqref{eq:M-alpha}.
\end{proof}

\subsection{Proofs of level-set inequalities on the idea of FMD}
\begin{lemma}\label{lem:A1} 
Let $\alpha \in [0,n)$ and $\mathcal{F}$, $\mathcal{G} \in L^1(\Omega; \mathbb{R}^+)$ 
satisfying the global comparison~\ref{ing:A3}. Assume that
\begin{align}\label{asmp:lem-A1}
\mathcal{V}_{\alpha}^{c}(\mathcal{F}; \kappa\lambda) \cap \Omega \neq \emptyset, \quad \mbox{ for some } \ \kappa, \lambda>0. 
\end{align}
 Then there exists a constant $C = C(n,\alpha)>0$ such that 
\begin{align}\label{ineq:A1}
d_{\mathcal{G}}^{\alpha}(\Omega; \sigma \lambda) \le C\left(\frac{\kappa}{\sigma}\right)^{\frac{n}{n-\alpha}} \mathrm{diam}(\Omega)^n, \quad \mbox{ for every} \ \sigma>0.
\end{align} 
\end{lemma}
\begin{proof}
Thanks to inequality~\eqref{eq:M-alpha} in Lemma~\ref{lem:M_alpha} and the global comparison~\eqref{eq:GC}, one has 
\begin{align} \nonumber 
d_{\mathcal{G}}^{\alpha}(\Omega; \sigma \lambda)  \le C_{n,\alpha} \left((\sigma \lambda)^{-1}\int_{\Omega}{\mathcal{G}(x) dx}\right)^{\frac{n}{n-\alpha}} \le C_{n,\alpha} \left((\sigma \lambda)^{-1}\int_{\Omega}{\mathcal{F}(x) dx}\right)^{\frac{n}{n-\alpha}}.
\end{align}
Due to~\eqref{asmp:lem-A1}, one can find $z_0 \in \Omega$ such that ${\mathbf{M}}_{\alpha}{\mathcal{F}}(z_0) \le \kappa \lambda$. Moreover, by the definition of fractional maximal function $\mathbf{M}_{\alpha}$, there holds 
\begin{align} \nonumber 
\int_{\Omega}{\mathcal{F}(x) dx} \le C_{n} D_0^n \fint_{B_{D_0}(z_0)}{\mathcal{F}(x) dx} \le C_{n} D_0^{n-\alpha}  {\mathbf{M}}_{\alpha}\mathcal{F}(z_0) \le C_{n} D_0^{n-\alpha} \kappa \lambda, 
\end{align}
where $D_0 = \mathrm{diam}(\Omega)$. Therefore we may conclude from two previous inequalities that
\begin{align*}
d_{\mathcal{G}}^{\alpha}(\Omega; \sigma \lambda)  &  \le C_{n,\alpha} \left(\frac{\kappa}{\sigma}\right)^{\frac{n}{n-\alpha}} D_0^n, 
\end{align*}
which leads to~\eqref{ineq:A1} and completes the proof.
\end{proof}

Moreover, it is worth highlighting some inequalities related to fractional-maximal distributions in this study. With these properties in hand, we will directly obtain the important level-set inequalities. They play an essential role in the description of our approach later.
\begin{lemma}\label{lem:A2}
Let $\alpha \in [0,n]$ and $\mathcal{G} \in L^1(\Omega; \mathbb{R}^+)$ such that 
\begin{align}\label{asmp:lem-A2}
\mathcal{V}_{\alpha}^{c}(\mathcal{G}; \lambda) \cap \Omega_{\varrho}(\xi)  \neq \emptyset, \quad \mbox{ for some }  \lambda, \varrho>0  \mbox{ and } \xi \in \Omega. 
\end{align}
Then for all $\sigma > 3^n$   there holds 
\begin{align}\label{eq:res11}
d_{\mathcal{G}}^{\alpha}(\Omega_{\varrho}(\xi); \sigma \lambda)  \le d_{\chi_{B_{2\varrho}(\xi)} \mathcal{G}}^{\alpha}(\Omega_{\varrho}(\xi); \sigma \lambda).
\end{align}
\end{lemma}
\begin{proof}
For any $\zeta \in B_{\varrho}(\xi)$, we can present $\mathbf{M}_{\alpha}\mathcal{G}$ as the maximum of two cutoff fractional maximal functions of $\mathcal{G}$ at level $\varrho>0$ defined in~\eqref{eq:frac-Malpha}, as follows 
\begin{align}\label{est:max}
{\mathbf{M}}_{\alpha}{\mathcal{G}}(\zeta)  & = \max \left\{ {\mathbf{M}}_{\alpha}^{\varrho}{\mathcal{G}}(\zeta); \ {\mathbf{T}}_{\alpha}^{\varrho}{\mathcal{G}}(\zeta) \right\}.
\end{align}
Furthermore, assumption~\eqref{asmp:lem-A2} allows us to find $z_1 \in \Omega_{\varrho}(\xi)$ satisfying ${\mathbf{M}}_{\alpha}{\mathcal{G}}(z_1) \le \lambda$. It is easy to check that 
$$B_{r}(\zeta) \subset B_{r+\varrho}(\xi) \subset B_{r+2\varrho}(z_1) \subset B_{3r}(z_1), \ \mbox{ for all } \ r \ge \varrho.$$ 
So we may estimate ${\mathbf{T}}_{\alpha}^{\varrho}{\mathcal{G}}$ by increasing the integral over $B_{r}(\zeta)$ to the one over $B_{3r}(z_1)$, one has
\begin{align}\nonumber
{\mathbf{T}}_{\alpha}^{\varrho}{\mathcal{G}}(\zeta) = \sup_{r \ge \varrho} \ r^{\alpha}{\fint_{B_{r}(\zeta)}{{\mathcal{G}}(x) dx}} & \le  \sup_{r \ge \varrho} \ \frac{\mathcal{L}^n(B_{3r}(z_1))}{\mathcal{L}^n(B_{r}(\zeta))} {r^{\alpha}  \fint_{B_{3r}(z_1)}{{\mathcal{G}}(x) dx}} \\ \nonumber
& \le 3^{n-\alpha} \sup_{r \ge \varrho} \ (3r)^{\alpha} {\fint_{B_{3r}(z_1)}{{\mathcal{G}}(x) dx}} \\  \label{est:T}
& \le 3^n {\mathbf{M}}_{\alpha}{\mathcal{G}}(z_1) \le 3^n \lambda.
\end{align} 
Substituting~\eqref{est:T} to~\eqref{est:max}, one obtains that
\begin{align}\nonumber
{\mathbf{M}}_{\alpha}{\mathcal{G}}(\zeta)  & = \max \left\{ \sup_{0 < r < \varrho} \ r^{\alpha}{\fint_{B_{r}(\zeta)}{\chi_{B_{2\varrho}(\xi)} {\mathcal{G}}(x) dx}}; \ {\mathbf{T}}_{\alpha}^{\varrho}{\mathcal{G}}(\zeta) \right\} \\ \label{eq:res9}
& \le \max \left\{ {\mathbf{M}}_{\alpha}^{\varrho}(\chi_{B_{2\varrho}(\xi)} \mathcal{G})(\zeta); \ 3^n \lambda \right\}, \ \mbox{ for all }  \zeta \in B_{\varrho}(\xi).
\end{align}
Here we emphasize that the first equality in~\eqref{eq:res9} comes from the fact that 
$$B_{r}(\zeta) \subset B_{2\varrho}(\xi), \mbox{ for all } r \in (0,\varrho).$$  
Finally,   we may conclude from~\eqref{eq:res9} that for all $\sigma > 3^n$ there holds
\begin{align}\label{eq:M-rho}
\mathcal{V}_{\alpha}(\mathcal{G}; \sigma \lambda) \cap \Omega_{\varrho}(\xi) = \left\{\zeta \in \Omega: \ \mathbf{M}_{\alpha}^{\varrho}(\chi_{B_{2\varrho}(\xi)} \mathcal{G})(\zeta)> \sigma \lambda\right\}\cap \Omega_{\varrho}(\xi),
\end{align}
which leads to inequality~\eqref{eq:res11}. 
\end{proof}

\begin{lemma}\label{lem:A3} 
Let $\gamma>1$, $\alpha \in [0,\frac{n}{\gamma})$ and two functions $\mathcal{F}$, $\mathcal{G}$ satisfy local comparison~\ref{ing:A2_1}. Then for any $\sigma > 3^n$, one can find $\kappa = \kappa(\sigma)>0$ such that if 
\begin{align}\label{asmp:lem-A3}
\mathcal{V}_{\alpha}^{c}(\mathcal{G}; \lambda) \cap \Omega_{\varrho}(\xi)  \neq \emptyset \ \mbox{ and } \ \mathcal{V}_{\alpha}^{c}(\mathcal{F}; \kappa \lambda) \cap \Omega_{\varrho}(\xi)  \neq \emptyset,
\end{align}
for some $\xi \in \Omega$ and $\varrho$, $\lambda \in \mathbb{R}^+$ then the following inequality
\begin{align}\label{eq:iigoal} 
d_{\mathcal{G}}^{\alpha}(\Omega_{\varrho}(\xi); \sigma \lambda)  \le C  \sigma^{-\frac{n \gamma}{n-\alpha \gamma}} \varrho^n,
\end{align}
holds. Here, the positive constant $C$ depends only on $n$, $\alpha$, $\gamma$, $\tilde{c}$.
\end{lemma}
\begin{proof}
If $B_{2\varrho}(\xi) \subset \Omega$ then we take $R = 2\varrho$ and $\nu = \xi$. Otherwise, if $B_{2\varrho}(\xi) \cap \Omega^c \neq \emptyset$, let us take $R = 4\varrho$ and $\nu \in \partial \Omega$ such that $|\xi- \nu| = \mathrm{dist}(x,\partial \Omega) \le 2 \varrho$. With this choice of $R$ and $\nu$, it is easily to check that $B_{2\varrho}(\xi) \subset B_{R}(\nu)$. The left-hand side of~\eqref{eq:iigoal} can be estimated by applying Lemma~\ref{lem:A2} with $\sigma > 3^n$ and using the definition of quasi-triangle triplet that $\mathcal{G} \le \tilde{c} (\varphi + \psi)$ in ~\eqref{cond:phi-psi}, it follows
\begin{align}\nonumber
d_{\mathcal{G}}^{\alpha}(\Omega_{\varrho}(\xi); \sigma \lambda) & \le d_{\chi_{B_{R}(\nu)}\mathcal{G}}^{\alpha}(\Omega_{\varrho}(\xi); \sigma \lambda) \\  \label{eq:estV-1}
& \le d_{\chi_{B_{R}(\nu)}\varphi}^{\alpha}(\Omega_{\varrho}(\xi); \tilde{c}^{-1} \sigma \lambda)   + d_{\chi_{B_{R}(\nu)}\psi}^{\alpha}(\Omega_{\varrho}(\xi); \tilde{c}^{-1} \sigma \lambda). 
\end{align}
To estimate two terms on the right hand side of~\eqref{eq:estV-1}, we apply Lemma~\ref{lem:M_alpha} with $s = 1$ and then $s = \gamma>1$, respectively. It is easy to rewrite in the form of the average integral as
\begin{align} \nonumber
d_{\mathcal{G}}^{\alpha}(\Omega_{\varrho}(\xi); \sigma \lambda) & \le  C\left(\frac{\tilde{c}}{\sigma \lambda}  R^n \fint_{B_{R}(\nu)}\psi(x) dx\right)^{\frac{n}{n-\alpha}} \\ \label{eq:estV-2a}
& \qquad \qquad + C \left(\left(\frac{\tilde{c}}{\sigma \lambda}\right)^{\gamma} R^n  \fint_{B_{R}(\nu)} |\varphi (x)|^{\gamma} dx \right)^{\frac{n}{n-\alpha \gamma}}.
\end{align}
Since $\varphi \in \mathcal{RH}^{\gamma}(\Omega)$ and the inequality $\varphi \le  \tilde{c} (\mathcal{G} + \psi)$ in~\eqref{cond:phi-psi}, there holds
\begin{align} \nonumber
\fint_{B_{R}(\nu)} |\varphi (x)|^{\gamma} dx & \le  C \left(\fint_{B_{2R}(\nu)} |\varphi (x)| dx\right)^{\gamma} \\ \label{eq:com-2c}
 & \le C \left(\tilde{c}\fint_{B_{2R}(\nu)} \mathcal{G}(x) dx + \tilde{c}\fint_{B_{2R}(\nu)} \psi(x) dx\right)^{\gamma}.
\end{align}
On the other hand, assumption~\eqref{asmp:lem-A3} ensures the existence of $z_1, z_2 \in \Omega_{\varrho}(\xi)$ satisfying ${\mathbf{M}}_{\alpha}{\mathcal{G}}(z_1) \le \lambda$ and ${\mathbf{M}}_{\alpha}{\mathcal{F}}(z_2) \le \kappa \lambda$. For this reason and noting that 
$$B_{2R}(\nu) \subset B_{3R}(\xi) \subset B_{3R + \varrho}(z_1) \cap B_{3R + \varrho}(z_2) \subset B_{4R}(z_1) \cap B_{4R}(z_2),$$ 
it gives us
\begin{align}\label{eq:com-2a}
\fint_{B_{2R}(\nu)}{\mathcal{G}(x)dx} & \le  2^n \fint_{B_{4R}(z_1)}{\mathcal{G}(x)dx}  \le 2^n (4R)^{-\alpha} {\mathbf{M}}_{\alpha} \mathcal{G}(x)(z_1) \le 2^n  R^{-\alpha} \lambda,
\end{align}
and similarly
\begin{align}\label{eq:com-2b}
\fint_{B_{2R}(\nu)}{\mathcal{F}(x)dx}  \le 2^n (4R)^{-\alpha} {\mathbf{M}}_{\alpha}\mathcal{F}(x)(z_2) \le 2^n R^{-\alpha} \kappa \lambda.
\end{align} 
Under the local comparison~\eqref{eq:LC}, we deduce from~\eqref{eq:com-2a} and~\eqref{eq:com-2b} that
\begin{align}\label{est:psi}
\fint_{B_{R}(\nu)}{\psi(x) dx} & \le 2^{n} (\varepsilon +c_{\varepsilon} \kappa) R^{-\alpha} \lambda, \quad \mbox{ for all } \ \varepsilon \in (0,1).
\end{align}
Combining between~\eqref{eq:estV-2a} with~\eqref{eq:com-2c},~\eqref{eq:com-2a} and~\eqref{est:psi}, it concludes that
\begin{align*}
d_{\mathcal{G}}^{\alpha}(\Omega_{\varrho}(\xi); \sigma \lambda) &\le  C\left[2^{n} \tilde{c} {\sigma}^{-1} (\varepsilon +c_{\varepsilon} \kappa) \right]^{\frac{n}{n-\alpha}} R^n  + C\left[2^{n}  \tilde{c}^2 {\sigma}^{-1} (1 + \varepsilon +c_{\varepsilon} \kappa) \right]^{\frac{n \gamma}{n-\alpha \gamma}} R^n, 
\end{align*}
which guarantees~\eqref{eq:iigoal} by taking $\varepsilon= \sigma^{-\frac{n(\gamma-1)}{n-\alpha\gamma}}$ and  $\kappa = \varepsilon c_{\varepsilon}^{-1}\in (0,\varepsilon)$.
\end{proof}

Having the ingredients and above-mentioned technical lemmas in mind, we are now ready to prove the first main theorem \ref{theo-A}. Our proof of Theorem \ref{theo-A} mainly relies on what obtained from Theorem~\ref{theo:G-lam-A} and Theorem~\ref{theo:G-lam-B} from below. 

In Theorem \ref{theo:G-lam-A}, we show that the fractional-maximal distribution inequality holds under hypotheses related to the local comparison~\ref{ing:A2_1} and the global one~\ref{ing:A3}. In particular, all parameters in this inequality still depend on a very small number $\varepsilon$. 

\begin{theorem} \label{theo:G-lam-A}
Let $\gamma>1$, $\alpha \in [0,\frac{n}{\gamma})$ and two functions $\mathcal{F}$, $\mathcal{G} \in L^1(\Omega; \mathbb{R}^+)$ satisfy both local comparison~\ref{ing:A2_1} and the global one~\ref{ing:A3}. Then there exists $\varepsilon_0 \in (0,1)$ such that the following inequality
\begin{align}\label{ineq:dG}
d_{\mathcal{G}}^{\alpha}(\Omega; \sigma \lambda)   \le C \varepsilon d_{\mathcal{G}}^{\alpha}(\Omega; \lambda)  + d_{\mathcal{F}}^{\alpha}(\Omega; \kappa \lambda),
\end{align}
holds for all $\lambda>0$ and $\varepsilon \in (0,\varepsilon_0)$,  with $\sigma = \varepsilon^{-\frac{n-\alpha \gamma}{n\gamma}}$ and $\kappa = \varepsilon c_{\varepsilon}^{-1}$.
\end{theorem}
\begin{proof}
Firstly, we will prove the following inequality
\begin{align}\label{G-lambda:1}
\mathcal{L}^n  \left( \mathcal{V}_{\alpha}(\mathcal{G}; \sigma \lambda) \cap \mathcal{V}_{\alpha}^{c}(\mathcal{F}; \kappa \lambda) \cap \Omega  \right)  \le C \varepsilon  \mathcal{L}^n  \left( \mathcal{V}_{\alpha}(\mathcal{G}; \lambda) \cap \Omega  \right),
\end{align}
with $\sigma = \varepsilon^{-\frac{n-\alpha \gamma}{n\gamma}}$ and $\kappa = \varepsilon c_{\varepsilon}^{-1}$. In order to obtain this inequality, we apply Lemma~\ref{lem:Cover} for two subsets defined by 
$$ \mathcal{P} = \mathcal{V}_{\alpha}(\mathcal{G}; \sigma \lambda) \cap \mathcal{V}_{\alpha}^{c}(\mathcal{F}; \kappa \lambda) \cap \Omega \quad \mbox{and} \quad \mathcal{Q} = \mathcal{V}_{\alpha}(\mathcal{G}; \lambda) \cap \Omega.$$ 
The proof consists in the construction of assumptions $i)$ and $ii)$ in Lemma~\ref{lem:Cover}. We first remark that~\eqref{G-lambda:1} obviously holds if $\mathcal{P}$ is empty. Hence we only need to consider the otherwise case $\mathcal{P} \neq \emptyset$ which ensures $i)$ from inequality~\eqref{ineq:A1} in Lemma~\ref{lem:A1} as follows
\begin{align}\label{est:G-lambda-1}
\mathcal{L}^n \left(\mathcal{P}\right) \le d_{\mathcal{G}}^{\alpha}(\Omega; \sigma \lambda) \stackrel{\eqref{ineq:A1}}{\le} C \left(\sigma^{-1} \varepsilon  c_{\varepsilon} ^{-1}\right)^{\frac{n}{n-\alpha}} \mathcal{L}^n\left(B_{r}(0)\right) \le  \varepsilon \mathcal{L}^n\left(B_{r}(0)\right).
\end{align}
Next, assumption $ii)$ will be proved by contradiction. More precisely, assume that $\Omega_{\varrho}(\xi) \cap \mathcal{Q}^c \neq \emptyset$ with  $\xi \in \Omega$ and $\varrho \in (0,r]$, it suffices to show that $\mathcal{L}^n\left(\mathcal{P} \cap B_{\varrho}(\xi)\right) \le \varepsilon \mathcal{L}^n\left(B_{\varrho}(\xi)\right)$. Indeed, without loss of generality we may again assume $\mathcal{P} \cap B_{\varrho}(\xi) \neq \emptyset$. Thanks to inequality~\eqref{eq:iigoal} in Lemma~\ref{lem:A3}, one gets that
\begin{align}\label{est:G-lambda-2}
\mathcal{L}^n\left(\mathcal{P} \cap B_{\varrho}(\xi)\right) & \stackrel{\eqref{eq:iigoal}}{\le}  C \left[ \left( {\sigma}^{-1} \varepsilon  \right)^{\frac{n}{n-\alpha}}   +  \sigma^{-\frac{n \gamma}{n-\alpha \gamma}} \right] \mathcal{L}^n\left(B_{\varrho}(\xi)\right) \le \varepsilon \mathcal{L}^n\left(B_{\varrho}(\xi)\right).
\end{align}
Here we remark that $(\sigma^{-1} \varepsilon)^{\frac{n}{n-\alpha}} = \varepsilon^{1+\frac{n}{\gamma(n-\alpha)}}$ and $c_{\varepsilon} >1$ in~\eqref{est:G-lambda-1} and~\eqref{est:G-lambda-2}. Therefore these inequalities hold for $\varepsilon \in (0,\varepsilon_0)$ where $\varepsilon_0$ small enough such that 
$$C \varepsilon_0^{\frac{n}{\gamma(n-\alpha)}} < 1 \ \mbox{ and } \ \varepsilon_0^{-\frac{n-\alpha \gamma}{n\gamma}} > 3^{n}.$$ 
Finally, to complete the proof one may decompose as follows
$$\mathcal{V}_{\alpha}(\mathcal{G}; \sigma \lambda) \cap \Omega = \mathcal{P} \cup \left( \mathcal{V}_{\alpha} (\mathcal{F}; \kappa \lambda) \cap \Omega\right),$$ 
which guarantees~\eqref{ineq:dG} by taking into account~\eqref{G-lambda:1}.
\end{proof}

A similar inequality to~\eqref{eq:iigoal} in Lemma~\ref{lem:A3} will be proposed in the following lemma which corresponds to the local comparison~\ref{ing:A2_2}. With the new inequality~\eqref{eq:iigoal-A4}, we can improve the fractional-maximal distribution inequality into a better version in which the parameter $\sigma$ is no longer depending on the small number $\varepsilon$.

\begin{lemma}\label{lem:A4} 
Let $\alpha \in \left[0,n\right)$ and two functions $\mathcal{F}$, $\mathcal{G}$ satisfy
local comparison~\ref{ing:A2_2} and the global one~\ref{ing:A3}. Then for every $\varepsilon \in (0,1)$, one can find $\sigma_0 =\sigma_0(n,\tilde{c})>0$ and $\kappa = \kappa({\varepsilon}) \in (0,\varepsilon)$ such that if provided~\eqref{asmp:lem-A3} for some $\xi \in \Omega$ and $\varrho$,~$\lambda>0$, then the following inequality
\begin{align}\label{eq:iigoal-A4} 
d_{\mathcal{G}}^{\alpha}(\Omega_{\varrho}(\xi); \sigma \lambda)  \le C  \left(\frac{\varepsilon}{\sigma} \right)^{\frac{n}{n-\alpha}} \varrho^n
\end{align}
holds for all $\sigma > \sigma_0$, where  $C = C(n, \alpha, \tilde{c})>0$.
\end{lemma}
\begin{proof}
As in the proof of Lemma~\ref{lem:A3}, we may chose $R = 2\varrho$ and $\nu = \xi$ if $B_{2\varrho}(\xi) \subset \Omega$ and otherwise, $R = 4\varrho$ and $\nu \in \partial \Omega$ such that $|\xi- \nu| = \mathrm{dist}(x,\partial \Omega) \le 2 \varrho$. Under assumption~\eqref{asmp:lem-A3}, one can find $z_1, z_2 \in \Omega_{\varrho}(\xi)$ such that ${\mathbf{M}}_{\alpha}{\mathcal{G}}(z_1) \le \lambda$ and ${\mathbf{M}}_{\alpha}{\mathcal{F}}(z_2) \le \kappa  \lambda$. Combining with the fact $B_{2R}(\nu) \subset  B_{4R}(z_1) \cap B_{4R}(z_2)$ to get that
\begin{align}\label{eq:com-2a-A4}
\fint_{B_{2R}(\nu)}{\mathcal{G}(x)dx} & \le   2^n  R^{-\alpha} \lambda \ \mbox{ and } \fint_{B_{2R}(\nu)}{\mathcal{F}(x)dx}  \le 2^n R^{-\alpha} \kappa  \lambda,
\end{align}
For every $\zeta \in \Omega_{\varrho}(\xi)$, one notes that $B_{r}(\zeta) \subset B_{r+\varrho}(\xi) \subset B_{2\varrho}(\xi) \subset B_{R}(\nu)$ for all $r \in (0,\varrho)$. The local comparison~\eqref{eq:Linf} and estimate~\eqref{eq:com-2a-A4} give us
\begin{align*}
\mathbf{M}_{\alpha}^{\varrho}\left(\chi_{B_R(\nu)}\varphi\right)(\zeta) \le R^{\alpha}  \sup_{x \in B_{R}(\nu)} \varphi(x) \le 2^n \hat{c} \left(1 + \kappa \right) \lambda \le 2^{n+1} \hat{c} \lambda,
\end{align*}
which yields that if $\tilde{c}^{-1} \sigma > 2^{n+1} \hat{c}$ then
\begin{align}\label{est:M-A4}
\mathcal{L}^n \left(\left\{\zeta \in \Omega_{\varrho}(\xi): \ \mathbf{M}_{\alpha}^{\varrho}\left(\chi_{B_R(\nu)}\varphi\right)(\zeta) > \tilde{c}^{-1} \sigma \lambda\right\} \right) = 0.
\end{align}
We note that the parameter $\kappa$ will be choose smaller than 1 later. Moreover, thanks to~\eqref{eq:M-rho} in the proof of Lemma~\ref{lem:A2} with $\sigma > 3^n$ and assumption $\mathcal{G} \le \tilde{c} (\varphi + \psi)$ in~\eqref{cond:phi-psi}, one has
\begin{align*}
d_{\mathcal{G}}^{\alpha}(\Omega_{\varrho}(\xi); \sigma \lambda) &\le  \mathcal{L}^n \left(\left\{\zeta \in \Omega_{\varrho}(\xi): \ \mathbf{M}_{\alpha}^{\varrho}\left(\chi_{B_R(\nu)}\mathcal{G} \right)(\zeta) >  \sigma \lambda\right\} \right) \\
& \le \mathcal{L}^n \left(\left\{\zeta \in \Omega_{\varrho}(\xi): \ \mathbf{M}_{\alpha}^{\varrho}\left(\chi_{B_R(\nu)}\varphi\right)(\zeta) > \tilde{c}^{-1} \sigma \lambda\right\} \right) \\
& \qquad  + \mathcal{L}^n \left(\left\{\zeta \in \Omega_{\varrho}(\xi): \ \mathbf{M}_{\alpha}^{\varrho}\left(\chi_{B_R(\nu)}\psi\right)(\zeta) > \tilde{c}^{-1} \sigma \lambda\right\} \right),
\end{align*}
which deduces from~\eqref{est:M-A4} that
\begin{align}\label{est:M-A4-b}
d_{\mathcal{G}}^{\alpha}(\Omega_{\varrho}(\xi); \sigma \lambda) \le d_{\chi_{B_R(\nu)}\psi}^{\alpha}(\Omega_{\varrho}(\xi); \tilde{c}^{-1} \sigma \lambda),
\end{align}
for all $\sigma > \sigma_0 = \max\left\{3^{n}, 2^{n+1}\tilde{c}\hat{c} \right\}$. Applying Lemma~\ref{lem:M_alpha} with $s = 1$ for the right-hand side of~\eqref{est:M-A4-b}, one gets that
\begin{align}\label{est:M-A4-c}
d_{\mathcal{G}}^{\alpha}(\Omega_{\varrho}(\xi); \sigma \lambda)  \le C \left(\frac{\tilde{c} R^n}{\sigma \lambda}  \fint_{B_R(\nu)} \psi(x) dx \right)^{\frac{n}{n-\alpha}}.
\end{align}
The integral term on the right-hand side of~\eqref{est:M-A4-c} can be estimated by using the local comparison~\eqref{eq:LC} and~\eqref{eq:com-2a-A4} as follows
\begin{align}\label{est:psi-A4}
\fint_{B_{R}(\nu)}{\psi(x) dx} & \le 2^{n} (\varepsilon + c_{\varepsilon} \kappa)  R^{-\alpha} \lambda \le 2^{n+1} \varepsilon R^{-\alpha} \lambda,
\end{align}
by choosing $\kappa = \varepsilon c_{\varepsilon}^{-1} \in (0, \varepsilon)$. Finally we may conclude~\eqref{eq:iigoal-A4} by collecting~\eqref{est:M-A4-c} and~\eqref{est:psi-A4}.
\end{proof}

\begin{theorem}\label{theo:G-lam-B}
Let $\alpha \in [0,n)$ and two functions $\mathcal{F}$, $\mathcal{G} \in L^1(\Omega; \mathbb{R}^+)$ satisfy local comparison~\ref{ing:A2_2} and the global one \ref{ing:A3}. Then for all $\lambda>0$ and $\varepsilon \in (0,1)$ there exists $\sigma_0 = \sigma_0(n,\tilde{c})>0$ and $\kappa  \in (0,\varepsilon)$ such that  
\begin{align}\label{ineq:dG-2}
d_{\mathcal{G}}^{\alpha}(\Omega; \sigma_0 \lambda) \le C \varepsilon d_{\mathcal{G}}^{\alpha}(\Omega; \lambda)  + d_{\mathcal{F}}^{\alpha}(\Omega; \kappa \lambda).
\end{align}
\end{theorem}
\begin{proof}
Inequality~\eqref{ineq:dG-2} can be proved by the same method as in the proof of Theorem~\ref{theo:G-lam-A} which is based on the covering Lemma~\ref{lem:Cover}, the only difference being in the applications of Lemmas~\ref{lem:A1} and~\ref{lem:A4}.
\end{proof}

\section{Bringing the norm back: Abstract results}\label{sec:abstract}
This section aims at offering some abstract results related comparisons for some general spaces, corresponding to the proof of Theorem \ref{theo-B}. It enables us to use recent FMD inequalities proved in Section~\ref{sec:ingredient} to prove certain norm inequalities in the setting of Lorentz, Orlicz spaces and Orlicz-Lorentz spaces, respectively.

In order to prove Theorem \ref{theo-B}, we give separate proofs for simplicity, through Theorem~\ref{theo:norm-L-1},~\ref{theo:norm-L-2}; Theorem~\ref{theo:Orlicz} and~\ref{theo:L-O} as follow.

\subsection{In Lorentz spaces}\label{sec:Lorentz_app}
It is well-known that the Lorentz spaces can be defined by using distribution function.
\begin{definition}[Lorentz spaces]\label{def:Lorentz}
Let $0<q<\infty$ and $0 <s \le \infty$. Lorentz space $L^{q,s}(\Omega)$ is the set of all Lebesgue measurable function $f$ on $\Omega$ such that $\|f\|_{L^{q,s}(\Omega)} < \infty$, where the $\|\cdot\|_{L^{q,s}(\Omega)}$ is the quasi-norm defined by
\begin{align*}
\|f\|_{L^{q,s}(\Omega)} :=  \left[ q \int_0^\infty{\left[\lambda^q d_f(\Omega;\lambda)\right]^{\frac{s}{q}} \frac{d\lambda}{\lambda}} \right]^{\frac{1}{s}}, 
\end{align*}
if $0< s < \infty$ and otherwise
\begin{align*}
\|f\|_{L^{q,\infty}(\Omega)} := \sup_{\lambda>0}{\left[\lambda^q d_f(\Omega;\lambda)\right]^{\frac{1}{q}}}.
\end{align*}
\end{definition}
\begin{remark}
When $q=s$, the Lorentz space $L^{q,q}(\Omega)$ coincides the classical Lebesgue space $L^q(\Omega)$.  Moreover, for $1<r<q<\infty$, one has the following relations
\begin{align*}
L^{q,1}(\Omega) \subset L^{q,q}(\Omega) = L^q(\Omega) \subset L^{q,\infty}(\Omega) \subset L^{r,1}(\Omega),
\end{align*}
and we refer to~\cite{Gra97} for further reading.
\end{remark}
\begin{theorem}[Lorentz norm estimate under~\ref{ing:A2_1}]\label{theo:norm-L-1}
Let $\gamma>1$, $\alpha \in [0,\frac{n}{\gamma})$ and two functions $\mathcal{F}$, $\mathcal{G} \in L^1(\Omega; \mathbb{R}^+)$ satisfy~\ref{ing:A2_1} and~\ref{ing:A3}. Then for every $0< q < \frac{n\gamma}{n-\alpha\gamma}$ and $0<s \le \infty$ there holds
\begin{align*}
\mathbf{M}_{\alpha}\mathcal{F} \in L^{q,s}(\Omega)\Longrightarrow \mathbf{M}_{\alpha}\mathcal{G} \in L^{q,s}(\Omega).
\end{align*}
More precisely, there exists a positive constant $C = C(n,\gamma,\alpha, q, s)$ such that 
\begin{align}\label{eq:main-A}
\|\mathbf{M}_{\alpha}\mathcal{G}\|_{L^{q,s}(\Omega)} & \le  C  \|\mathbf{M}_{\alpha}\mathcal{F}\|_{L^{q,s}(\Omega)}.
\end{align}
\end{theorem}
\begin{proof}
According to Theorem~\ref{theo:G-lam-A}, one can find $\varepsilon_0>0$ small enough such that for all $\varepsilon \in (0,\varepsilon_0)$ and $\lambda>0$, there holds
\begin{align}\label{est:A-1}
d_{\mathcal{G}}^{\alpha}(\Omega; \sigma \lambda)   \le C \varepsilon d_{\mathcal{G}}^{\alpha}(\Omega; \lambda)  + d_{\mathcal{F}}^{\alpha}(\Omega; \kappa \lambda),
\end{align}
where $\sigma = \varepsilon^{-\frac{n-\alpha \gamma}{n\gamma}}$ and $\kappa = \varepsilon c_{\varepsilon}^{-1}$. By changing of variable $\lambda$ to $\delta \lambda$, the norm of $\mathbf{M}_{\alpha}\mathcal{G}$ in Lorentz space $L^{q,s}(\Omega)$ can be rewritten as
\begin{align}\label{est:A-2}
\|\mathbf{M}_{\alpha}\mathcal{G}\|^s_{L^{q,s}(\Omega)}  = \delta^s q\int_0^\infty{ \left[\lambda^q d_{\mathcal{G}}^{\alpha}(\Omega; \delta \lambda)\right]^{\frac{s}{q}}\frac{d\lambda}{\lambda}}, \quad \forall \delta>0.
\end{align}
Thanks to~\eqref{est:A-2} and~\eqref{est:A-1}, one obtains that
\begin{align}\nonumber
\|\mathbf{M}_{\alpha}\mathcal{G}\|^s_{L^{q,s}(\Omega)}  & \stackrel{\eqref{est:A-2}}{=}  \sigma^{s}  q\int_0^\infty{\left[\lambda^q d_{\mathcal{G}}^{\alpha}(\Omega; \sigma \lambda)\right]^{\frac{s}{q}}\frac{d\lambda}{\lambda}}\\ \nonumber
& \stackrel{\eqref{est:A-1}}{\le}  C \sigma^{s} \varepsilon^{\frac{s}{q}} q\int_0^\infty{\left[\lambda^q d_{\mathcal{G}}^{\alpha}(\Omega; \lambda)\right]^{\frac{s}{q}}\frac{d\lambda}{\lambda}} + C \sigma^{s} q\int_0^\infty{\left[\lambda^q d_{\mathcal{F}}^{\alpha}(\Omega; \kappa \lambda)\right]^{\frac{s}{q}}\frac{d\lambda}{\lambda}} \\ \label{est:5b}
& \stackrel{\eqref{est:A-2}}{=}  C \sigma^{s} \varepsilon^{\frac{s}{q}} \|\mathbf{M}_{\alpha}\mathcal{G}\|^s_{L^{q,s}(\Omega)} + C  \sigma^{s}  \kappa^{-s} \|\mathbf{M}_{\alpha}\mathcal{F}\|^s_{L^{q,s}(\Omega)}.
\end{align}
For every $0<s<\infty$ and $0<q<\frac{n\gamma}{n-\alpha\gamma}$, we may choose $\varepsilon \in (0,\varepsilon_0)$ in~\eqref{est:5b} satisfying
\begin{align*}
C \sigma^{s} \varepsilon^{\frac{s}{q}} = C \varepsilon^{s\left(\frac{1}{q}-\frac{n-\alpha\gamma}{n\gamma}\right)} \le \frac{1}{2},
\end{align*}
to obtain~\eqref{eq:main-A}. The same conclusion can be drawn for the case $s = \infty$. 
\end{proof}

\begin{theorem}[Lorentz norm estimate under~\ref{ing:A2_2}]\label{theo:norm-L-2}
Let $\alpha \in [0,n)$ and two functions $\mathcal{F}$, $\mathcal{G} \in L^1(\Omega; \mathbb{R}^+)$ satisfy~\ref{ing:A2_2} and~\ref{ing:A3}. Then for all $0< q < \infty$ and $0<s \le \infty$, there exists a positive $C = C(n,\gamma,\alpha,q,s)$ such that the following norm estimate
\begin{align}\label{eq:main-B}
\|\mathbf{M}_{\alpha}\mathcal{G}\|_{L^{q,s}(\Omega)} & \le  C  \|\mathbf{M}_{\alpha}\mathcal{F}\|_{L^{q,s}(\Omega)}.
\end{align}
\end{theorem}
\begin{proof}
Let us perform the same technique as in the proof of Theorem~\ref{theo:norm-L-2} using inequality~\eqref{ineq:dG-2}, we obtain that
\begin{align}\nonumber
\|\mathbf{M}_{\alpha}\mathcal{G}\|^s_{L^{q,s}(\Omega)}  \le C \sigma_0^{s} \varepsilon^{\frac{s}{q}} \|\mathbf{M}_{\alpha}\mathcal{G}\|^s_{L^{q,s}(\Omega)} + C  \sigma_0^{s}  \kappa^{-s} \|\mathbf{M}_{\alpha}\mathcal{F}\|^s_{L^{q,s}(\Omega)}.
\end{align}
Therefore we may choose $\varepsilon \in (0,1)$ in this inequality such that $C \sigma_0^{s} \varepsilon^{\frac{s}{q}}  \le \frac{1}{2}$ to get~\eqref{eq:main-B} for any $0<s<\infty$ and $0<q<\infty$. We also obtain the same result for the case $s = \infty$. 
\end{proof}

\subsection{In Orlicz spaces}\label{sec:Orlicz_app}
The study of Orlicz norm estimate in this section is of our independent interest. In this section we will make theory of FMD also be available in the framework of Orlicz spaces. Let us first briefly recall the definitions as well as some basic results concerning to this space, which are required to prove if necessary.

\begin{definition}[Young function]\label{def:youngfunc}
Let $\Phi$ be a non-negative, increasing and convex real-valued function on $[0,\infty)$. We say that $\Phi$ is the Young function if 
\begin{align}\label{eq:limit_young}
\lim_{\mu \to 0^+}{\frac{\Phi(\mu)}{\mu}} = 0, \lim_{\mu \to \infty}{\frac{\Phi(\mu)}{\mu}} = \infty.
\end{align}
\end{definition}

\begin{definition}[$\Delta_2$ condition]
The Young function $\Phi$ is said to satisfy the global $\Delta_2$ condition, denoted by $\Phi \in \Delta_2$ if there exists $\tau_1 \ge 2$ such that
\begin{align*}
\Phi(2\mu) \le \tau_1\Phi(\mu), \ \ \text{for all} \ \mu\ge 0.
\end{align*}
\end{definition}
We notice that the $\Delta_2$ condition is equivalent to the fact that for every $a>1$, there exists $\tau_1(a)>0$ such that $\Phi(a\mu) \le \tau_1(a) \Phi(\mu)$ for all $\mu \ge 0$. This fact is stated in the following lemma. We can refer to \cite[Lemma 2.2.7]{Hasto}, the work was done by H{\"a}st{\"o}.

\begin{lemma}\label{lem:Delta2_pro2}
Let $\Phi$ be a Young function. Then, $\Phi \in \Delta_2$ if and only if there exist two constants $K_1>0$ and $p_1>1$ such that for any $a > 1$ and $\mu>0$, there holds
\begin{align}\label{eq:prop2}
\Phi(a\mu) \le K_1 a^{p_1} \Phi(\mu).
\end{align}
\end{lemma}

\begin{definition}[Orlicz spaces]
Let a bounded open subset $\Omega \subset \mathbb{R}^n$ and the Young function $\Phi \in \Delta_2$. The Orlicz class $\mathcal{O}^{\Phi}(\Omega)$ is defined to be the set all of measurable functions $f:\Omega \to \mathbb{R}$ satisfying 
$$\int_\Omega{\Phi(|f(x)|)dx} < \infty.$$
The Orlicz space $L^\Phi(\Omega)$ is the smallest linear space containing $\mathcal{O}^{\Phi}(\Omega)$, endowed with the Luxemburg norm
\begin{align*}
\|f\|_{L^\Phi(\Omega)} = \inf\left\{\tau >0 : \ \int_\Omega{\Phi\left(\frac{|f(x)|}{\tau} \right) dx} \le 1 \right\}.
\end{align*}
\end{definition}
For further reading on the theory of Orlicz spaces, we address the reader to~\cite{Orlicz1932,Rao1991,Hasto} with many references given therein.

\begin{definition}[$\nabla_2$ condition]
The Young function $\Phi$ is said to satisfy the $\nabla_2$ condition, $\Phi \in \nabla_2$, if there exists $\tau_2>1$ such that
\begin{align*}
\Phi(\mu) \le \frac{\Phi(\tau_2 \mu)}{2\tau_2}.
\end{align*}
\end{definition}

The following property of Young function in $\nabla_2$ is similar to Lemma~\ref{lem:Delta2_pro2}, see~\cite{Hasto} for its proof.

\begin{lemma}\label{lem:nabla2}
Let $\Phi$ be a Young function. Then, $\Phi \in \nabla_2$ if and only if there exist two constants $K_2>0$ and $p_2>1$ such that for any $a \in (0,1)$ and $\mu>0$ there holds
\begin{align}\label{eq:nabla2}
\Phi(a \mu) \le K_2 a^{p_2} \Phi(\mu).
\end{align}
\end{lemma}

It is noticeable that the Young function $\Phi$ satisfies both the $\Delta_2$ and $\nabla_2$ conditions, often denoted by $\Phi \in \Delta_2 \cap \nabla_2$, ensures that the Young function $\Phi$ grows neither too fast nor too slow. Indeed, the limits in \eqref{eq:limit_young} combines with two conditions $\Delta_2$ and $\nabla_2$ gives
\begin{align*}
0 = \Phi(0) = \lim_{\mu\to 0^+}{\Phi(\mu)}, \quad \text{and} \quad  \lim_{\mu \to \infty}{\Phi(\mu)} = \infty,
\end{align*}
to obtain the limits are not too fast or too slow as $\mu \to 0^+$ and/or $\mu \to \infty$. The assertion of following Lemma shows us the fact that Orlicz class $\mathcal{O}^{\Phi}(\Omega)$ is not different from the the Orlicz space $L^\Phi(\Omega)$ in the case of Young function $\Phi \in \Delta_2 \cap \nabla_2$.

\begin{lemma}\label{lem:norm-O}
Let Young function $\Phi \in \Delta_2 \cap \nabla_2$. One can find a constant $C \ge 1$ such that
\begin{align}\label{eq:norm-O}
C^{-1} \left(\|f\|_{L^{\Phi}(\Omega)}^{p_2} - 1\right) \le \int_{\Omega} \Phi(|f(x)|) dx \le C  \left(\|f\|_{L^{\Phi}(\Omega)}^{p_1} + 1\right),
\end{align}
for all $f \in L^{\Phi}(\Omega)$, where $p_1 \ge p_2 >1$ are given as in Lemma~\ref{lem:Delta2_pro2} and Lemma~\ref{lem:nabla2}.
\end{lemma}
\begin{proof}
For every $f \in L^{\Phi}(\Omega)$, let us set 
\begin{align*}
\Gamma^f = \left\{\tau >0 : \ \int_\Omega{\Phi\left(\frac{|f(x)|}{\tau} \right) dx} \le 1 \right\}.
\end{align*}
It is easy to see that if $1 \in \Gamma^f$ then $\int_{\Omega} \Phi(|f(x)|) dx \le 1$ and $\|f\|_{L^{\Phi}(\Omega)} \le 1$ which imply to~\eqref{eq:norm-O} with $C = 1$. Otherwise, if $1 \not \in \Gamma^f$ then $\int_{\Omega} \Phi(|f(x)|) dx > 1$ which follows that
\begin{align*}
1 < \int_{\Omega} \Phi(|f(x)|) dx \le  \int_\Omega{\Phi\left(\frac{|f(x)|}{\tau} \right) dx}, \quad \forall \tau \le 1.
\end{align*}
That means $\Gamma^f \subset (1,\infty)$. In this case, without loss of generality we may assume that there exists a decreasing sequence $(\tau_k)_{k \in \mathbb{N}} \subset \Gamma^f \cap (1,\infty)$ such that $\tau_k \to \|f\|_{L^{\Phi}(\Omega)}$ as $k$ tends to infinity. Since $\Phi \in \Delta_2 \cap \nabla_2$, applying inequalities~\eqref{eq:prop2} in Lemma~\ref{lem:Delta2_pro2} and~\eqref{eq:nabla2} in Lemma~\ref{lem:nabla2} for every $k \in \mathbb{N}$, one has
\begin{align*}
K_2^{-1} \tau_k^{p_2}\int_{\Omega} \Phi\left(\frac{|f(x)|}{\tau_k}\right) dx\le \int_{\Omega} \Phi(|f(x)|) dx \le K_1 \tau_k^{p_1} \int_\Omega{\Phi\left(\frac{|f(x)|}{\tau_k}  \right) dx} \le  K_1 \tau_k^{p_1}.
\end{align*}
Sending $k \to \infty$ in this inequality, one obtains that
\begin{align*}
K_2^{-1}\|f\|_{L^{\Phi}(\Omega)}^{p_2} \le \int_{\Omega} \Phi(|f(x)|) dx \le K_1 \|f\|_{L^{\Phi}(\Omega)}^{p_1},
\end{align*}
which implies to~\eqref{eq:norm-O}. 
\end{proof}

\begin{remark}
Let $\Phi$ be a Young function belonging to $\Delta_2 \cap \nabla_2$. Then, for $p_1 \ge p_2 >1$ given as in Lemma~\ref{lem:Delta2_pro2} and Lemma~\ref{lem:nabla2}, one has
\begin{align*}
L^{p_1}(\Omega) \subset L^\Phi(\Omega) \subset L^{p_2}(\Omega) \subset L^1(\Omega).
\end{align*}
\end{remark}
\begin{remark}
It states that the Orlicz space $L^{\Phi}$ generalizes $L^p$ space for $p>1$ in the sense that when $\Phi(\mu) = \mu^p$, $\mu \ge 0$ is the Young function, $\Phi \in \Delta_2 \cap \nabla_2$ and then it becomes the special case Lebesgue space, i.e. $L^\Phi(\Omega) = L^p(\Omega)$.
\end{remark}

\begin{theorem}[Orlicz norm estimates]
\label{theo:Orlicz}
Let Young function $\Phi \in \Delta_2$ and $p_1 >1$ given as in Lemma~\ref{lem:Delta2_pro2}. Assume that two functions $\mathcal{F}$, $\mathcal{G} \in L^1(\Omega; \mathbb{R}^+)$ satisfy~\ref{ing:A2_1} and~\ref{ing:A3}. Then for every $\alpha \in [0,\frac{n}{\gamma})$ satisfying $\alpha > n\left(\frac{1}{\gamma} - \frac{1}{p_1}\right)$, if $\mathbf{M}_{\alpha}\mathcal{F} \in L^{\Phi}(\Omega)$ then $\mathbf{M}_{\alpha}\mathcal{G} \in L^{\Phi}(\Omega)$ corresponding the following estimate
\begin{align}\label{eq:theo-Orlicz-2}
\|\mathbf{M}_{\alpha}\mathcal{G}\|_{L^{\Phi}(\Omega)} \le C_0 \|\mathbf{M}_{\alpha}\mathcal{F}\|_{L^{\Phi}(\Omega)}.
\end{align}
Moreover, under the local comparison~\ref{ing:A2_2}, the inequality~\eqref{eq:theo-Orlicz-2} holds for all $\alpha \in [0,n)$.
\end{theorem}
\begin{proof}
Thanks to Theorem~\ref{theo:G-lam-A}, one can find $\varepsilon_0>0$ such that for all $\varepsilon \in (0,\varepsilon_0)$ and $\lambda>0$, there holds
\begin{align}\label{est:O-1}
d_{\mathcal{G}}^{\alpha}(\Omega; \sigma \lambda)   \le C \varepsilon d_{\mathcal{G}}^{\alpha}(\Omega; \lambda)  + d_{\mathcal{F}}^{\alpha}(\Omega; \kappa \lambda),
\end{align}
where $\sigma = \varepsilon^{-\frac{n-\alpha \gamma}{n\gamma}}$ and $\kappa = \varepsilon c_{\varepsilon}^{-1}$. For every $\mu>0$, let us replace $\lambda$ by $\sigma^{-1} \Phi^{-1}(\mu)$ in~\eqref{est:O-1}, this fractional-maximal distribution inequality can be rewritten as
\begin{align}\label{est:O-2}
d_{\mathcal{G}}^{\alpha}(\Omega;  \Phi^{-1}(\mu))   \le C \varepsilon d_{\mathcal{G}}^{\alpha}(\Omega; \sigma^{-1}\Phi^{-1}(\mu))  + d_{\mathcal{F}}^{\alpha}(\Omega; \kappa \sigma^{-1}\Phi^{-1}(\mu)).
\end{align} 
The inequality~\eqref{eq:prop2} in Lemma~\ref{lem:Delta2_pro2} gives us
\begin{align*}
\Phi(\sigma\mathbf{M}_{\alpha}\mathcal{G}) \le K_1 \sigma^{p_1} \Phi(\mathbf{M}_{\alpha}\mathcal{G}), 
\end{align*}
with notice that $\sigma >1$ for all $\varepsilon \in (0,1)$. This fact implies that
\begin{align*}
d_{\mathcal{G}}^{\alpha}(\Omega; \sigma^{-1}\Phi^{-1}(\mu)) & = d_{\Phi(\sigma\mathbf{M}_{\alpha}\mathcal{G})}(\Omega; \mu) \\
& \le d_{K_1 \sigma^{p_1} \Phi(\mathbf{M}_{\alpha}\mathcal{G})}(\Omega; \mu) = d_{\Phi(\mathbf{M}_{\alpha}\mathcal{G})}(\Omega; K_1^{-1} \sigma^{-p_1} \mu),
\end{align*}
and similarly
\begin{align*}
d_{\mathcal{F}}^{\alpha}(\Omega; \kappa \sigma^{-1}\Phi^{-1}(\mu)) \le d_{\Phi(\mathbf{M}_{\alpha}\mathcal{F})}(\Omega; \kappa K_1^{-1} \sigma^{-p_1} \mu),
\end{align*}
which with combining~\eqref{est:O-2} one has
\begin{align}\label{est:O-3}
d_{\Phi(\mathbf{M}_{\alpha}\mathcal{G})}(\Omega; \mu)  \le C \varepsilon d_{\Phi(\mathbf{M}_{\alpha}\mathcal{G})}(\Omega; K_1^{-1} \sigma^{-p_1} \mu) + d_{\Phi(\mathbf{M}_{\alpha}\mathcal{F})}(\Omega; K_1^{-1} \kappa^{p_1} \sigma^{-p_1} \mu)   .
\end{align}
Inequality~\eqref{est:O-3} leads to
\begin{align}\nonumber
\int_{\Omega} \Phi(\mathbf{M}_{\alpha}\mathcal{G})(x) dx & = \int_0^{\infty} d_{\Phi(\mathbf{M}_{\alpha}\mathcal{G})}(\Omega; \mu)  d\mu \\ \nonumber
& \le C \varepsilon \int_0^{\infty} d_{\Phi(\mathbf{M}_{\alpha}\mathcal{G})}(\Omega; K_1^{-1} \sigma^{-p_1} \mu) d\mu + \int_0^{\infty} d_{\Phi(\mathbf{M}_{\alpha}\mathcal{F})}(\Omega; K_1^{-1} \kappa^{p_1} \sigma^{-p_1} \mu)  d\mu,
\end{align}
which implies to the following estimate by changing variables in the integrals
\begin{align}\nonumber
\int_{\Omega} \Phi(\mathbf{M}_{\alpha}\mathcal{G})(x) dx & \le C K_1 \sigma^{p_1} \varepsilon \int_0^{\infty} d_{\Phi(\mathbf{M}_{\alpha}\mathcal{G})}(\Omega;  \mu) d\mu \\ \nonumber
& \hspace{3cm} + K_1\kappa^{-p_1} \sigma^{p_1}\int_0^{\infty} d_{\Phi(\mathbf{M}_{\alpha}\mathcal{F})}(\Omega;  \mu)  d\mu \\ \nonumber
& \le C K_1  \varepsilon^{1-p_1\left(\frac{1}{\gamma}-\frac{\alpha}{n}\right)} \int_{\Omega} \Phi(\mathbf{M}_{\alpha}\mathcal{G})(x) dx \\ \label{est:O-4}
& \hspace{3cm} + K_1 \kappa^{-p_1} \sigma^{p_1}\int_{\Omega} \Phi(\mathbf{M}_{\alpha}\mathcal{F})(x) dx.
\end{align}
For all $\alpha \in [0,\frac{n}{\gamma})$ and $\alpha > n\left(\frac{1}{\gamma} - \frac{1}{p_1}\right)$, one can choose $\varepsilon = \min\{\varepsilon_0, \varepsilon_1\}$ in~\eqref{est:O-4} with $\varepsilon_1$ is small enough such that $ C K_1  \varepsilon_1^{1-p_1\left(\frac{1}{\gamma}-\frac{\alpha}{n}\right)} \le \frac{1}{2}$ to observe that
\begin{align}\label{eq:theo-Orlicz-1}
\int_{\Omega} \Phi(\mathbf{M}_{\alpha}\mathcal{G})(x) dx  & \le C_0 \int_{\Omega} \Phi(\mathbf{M}_{\alpha}\mathcal{F})(x) dx.
\end{align}
with constant $C_0 =  \max\left\{1,K_1 \varepsilon_1^{-p_1} c_{\varepsilon_1}^{p_1} \varepsilon_1^{p_1\left(\frac{\alpha}{n} - \frac{1}{\gamma}\right)}\right\} \ge 1$.\\

For every $\tau >0$,  scaling $\lambda$ by $\frac{\lambda}{\tau}$ in level-set inequality on fractional-maximal distribution~\eqref{est:O-1}, one may prove a similar version of inequality~\eqref{eq:theo-Orlicz-1} corresponding to $\frac{\mathbf{M}_{\alpha}\mathcal{G}}{\tau}$ and $\frac{\mathbf{M}_{\alpha}\mathcal{F}}{\tau}$ as follows
\begin{align}\label{eq:theo-Orlicz-1b}
\int_{\Omega} \Phi\left(\frac{\mathbf{M}_{\alpha}\mathcal{G}}{\tau}\right)dx  & \le C_0 \int_{\Omega} \Phi\left(\frac{\mathbf{M}_{\alpha}\mathcal{F}}{\tau}\right)dx.
\end{align}
Next we are going to prove~\eqref{eq:theo-Orlicz-2}. Let us introduce two sets as follows
\begin{align}\nonumber 
\Gamma^{f} = \left\{\tau >0: \ \int_{\Omega} \Phi\left(\frac{\mathbf{M}_{\alpha}f}{\tau}\right)dx \le 1 \right\}, 
\end{align}
where $f = \mathcal{F}$ or $f = \mathcal{G}$. For every $\tau \in \Gamma^{\mathcal{F}}$,   we apply~\eqref{eq:theo-Orlicz-1b} and the convexity of $\Phi$ with the fact $C_0 \ge 1$ to conclude that
\begin{align*}
\int_{\Omega} \Phi\left(\frac{\mathbf{M}_{\alpha}\mathcal{G}}{C_0\tau}\right)dx \le \frac{1}{C_0} \int_{\Omega} \Phi\left(\frac{\mathbf{M}_{\alpha}\mathcal{G}}{\tau}\right)dx \le \int_{\Omega} \Phi\left(\frac{\mathbf{M}_{\alpha}\mathcal{F}}{\tau}\right)dx \le 1,
\end{align*}
which ensure that $C_0\tau \in \Gamma^{\mathcal{G}}$. We conclude that $\Gamma^{\mathcal{F}} \subset \frac{1}{C_0} \Gamma^{\mathcal{G}}$ and hence $\inf \Gamma^{\mathcal{F}} \ge \frac{1}{C_0} \inf\Gamma^{\mathcal{G}}$ which completes the proof of~\eqref{eq:theo-Orlicz-2}. Finally,  under the local comparison~\ref{ing:A2_2}, the distribution inequality~\eqref{est:O-1} is valid even for constant $\sigma = \sigma_0$ which does not depend on $\varepsilon$. It follows that the exponent of $\varepsilon$ in~\eqref{est:O-4} is only 1. Therefore~\eqref{eq:theo-Orlicz-1} holds for all $\alpha \in [0,n)$ in this case.
\end{proof}

\subsection{In Orlicz-Lorentz spaces}\label{subsec:L-O}
Being a natural generalization of Orlicz and Lorentz spaces, the so-called Orlicz-Lorentz spaces are designed to glue properties in both Orlicz and Lorentz spaces in some sense. These spaces have a rich structure that has been constructed in different ways,   see \cite{MontOL,Kaminska}. Here, it is also certainly pleasing to deal with Orlicz-Lorentz norm estimates in our work.

In the sequel, we remark that definition of Orlicz-Lorentz spaces will be presented in the language of FMD. 
\begin{definition}[Orlicz-Lorentz space]
Let Young function $\Phi \in \Delta_2$ and two parameters $0< q < \infty$ and $0< s \le \infty$. The Orlicz-Lorentz class $\mathcal{O}^{\Phi}(q,s)(\Omega)$ contains all of measurable functions $f$ such that $$\|\Phi(|f|)\|_{L^{q,s}(\Omega)} < \infty.$$
The Orlicz-Lorentz space $L^{\Phi}(q,s)(\Omega)$ is defined as the linear hull of the Orlicz-Lorentz class $\mathcal{O}^{\Phi}(q,s)(\Omega)$, endowed with the following norm
 \begin{align*}
 \|f\|_{L^{\Phi}(q,s)(\Omega)} = \inf \left\{\tau >0: \ \left\|\Phi\left(\frac{|f|}{\tau}\right)\right\|_{L^{q,s}(\Omega)} \le 1\right\}.
 \end{align*}
\end{definition}
We recall here the quasi-norm of $\Phi(|f|)$ in Lorentz spaces $L^{q,s}(\Omega)$ is given by
\begin{align*}
\|\Phi(|f|)\|_{L^{q,s}(\Omega)} = \left[\int_0^{\infty} q\left[ \mu^q d_{\Phi(|f|)}(\Omega; \mu)\right]^{\frac{s}{q}} \frac{d\mu}{\mu}\right]^{\frac{1}{s}},
\end{align*}
if $s < \infty$ and otherwise
\begin{align*}
\|\Phi(|f|)\|_{L^{q,\infty}(\Omega)} = \sup_{\mu >0} \left[\mu^q d_{\Phi(|f|)}(\Omega; \mu)\right]^{\frac{1}{q}}.
\end{align*}

The Orlicz-Lorentz spaces (equipped with the Luxemburg norm and the norm in Lorentz spaces) also have a lot of interesting properties, thus one can expect the more general and beautiful results in this case.

\begin{theorem}[Orlicz-Lorentz norm estimates]\label{theo:L-O}
Let Young function $\Phi \in \Delta_2$ and $p_1>1$ given as in Lemma~\ref{lem:Delta2_pro2}. Assume that two functions $\mathcal{F}$, $\mathcal{G} \in L^1(\Omega; \mathbb{R}^+)$ satisfy~\ref{ing:A2_1} and \ref{ing:A3}. Then for every $\alpha \in [0,\frac{n}{\gamma})$, $0 < q < \frac{n\gamma}{p_1(n - \gamma \alpha)}$ and $0< s \le \infty$, if $\mathbf{M}_{\alpha}\mathcal{F} \in L^{\Phi}(q,s)(\Omega)$ then $\mathbf{M}_{\alpha}\mathcal{G} \in L^{\Phi}(q,s)(\Omega)$ corresponding the following estimate
\begin{align}\label{eq:L-O}
\|\mathbf{M}_{\alpha}\mathcal{G}\|_{L^{\Phi}(q,s)(\Omega)} \le C \|\mathbf{M}_{\alpha}\mathcal{F}\|_{L^{\Phi}(q,s)(\Omega)}.
\end{align}
Moreover, under local comparison~\ref{ing:A2_2}, the inequality~\eqref{eq:L-O} holds for all $\alpha \in [0,n)$, $0< q < \infty$ and $0<s \le \infty$.
\end{theorem}
\begin{proof}
Similar to the proof of Theorem~\ref{theo:Orlicz}, the following estimate
\begin{align}\label{est:L-O-3}
d_{\Phi(\mathbf{M}_{\alpha}\mathcal{G})}(\Omega; \mu)  \le C \varepsilon d_{\Phi(\mathbf{M}_{\alpha}\mathcal{G})}(\Omega; K_1^{-1} \sigma^{-p_1} \mu) + d_{\Phi(\mathbf{M}_{\alpha}\mathcal{F})}(\Omega; K_1^{-1} \kappa^{p_1} \sigma^{-p_1} \mu)   ,
\end{align}
holds for every $\mu>0$ and $\varepsilon \in (0,\varepsilon_0)$ with $\sigma = \varepsilon^{-\frac{n-\alpha \gamma}{n\gamma}}$ and $\kappa = \varepsilon c_{\varepsilon}^{-1}$. We infer from~\eqref{est:L-O-3} that
\begin{align}\nonumber
\|\Phi(\mathbf{M}_{\alpha}\mathcal{G})\|_{L^{q,s}(\Omega)}^s & = \int_0^{\infty} q\left[ \mu^q d_{\Phi(\mathbf{M}_{\alpha}\mathcal{G})}(\Omega; \mu)\right]^{\frac{s}{q}} \frac{d\mu}{\mu}  \\ \nonumber
& \le C \varepsilon^{\frac{s}{q}} \int_0^{\infty} q\left[\mu^q d_{\Phi(\mathbf{M}_{\alpha}\mathcal{G})}(\Omega; K_1^{-1} \sigma^{-p_1} \mu) \right]^{\frac{s}{q}} \frac{d\mu}{\mu} \\ \nonumber
& \qquad \qquad + C\int_0^{\infty} q\left[\mu^q d_{\Phi(\mathbf{M}_{\alpha}\mathcal{F})}(\Omega; K_1^{-1} \kappa^{p_1} \sigma^{-p_1} \mu)\right]^{\frac{s}{q}} \frac{d\mu}{\mu},
\end{align}
which is equivalent to the following estimate by changing variables
\begin{align}\nonumber
\|\Phi(\mathbf{M}_{\alpha}\mathcal{G})\|_{L^{q,s}(\Omega)}^s 
& \le C \varepsilon^{\frac{s}{q}} \left(K_1^{-1} \sigma^{-p_1}\right)^{-s} \int_0^{\infty} q\left[\mu^q d_{\Phi(\mathbf{M}_{\alpha}\mathcal{G})}(\Omega;  \mu) \right]^{\frac{s}{q}} \frac{d\mu}{\mu} \\ \nonumber
& \qquad \qquad + C \left(K_1^{-1} \kappa^{p_1} \sigma^{-p_1}\right)^{-s} \int_0^{\infty}   q\left[\mu^q d_{\Phi(\mathbf{M}_{\alpha}\mathcal{F})}(\Omega; \mu)\right]^{\frac{s}{q}} \frac{d\mu}{\mu} \\ \nonumber
& \le C K_1^{s} \sigma^{sp_1} \varepsilon^{\frac{s}{q}} \|\Phi(\mathbf{M}_{\alpha}\mathcal{G})\|_{L^{q,s}(\Omega)}^s + K_1^s \kappa^{sp_1} \sigma^{sp_1} \|\Phi(\mathbf{M}_{\alpha}\mathcal{F})\|_{L^{q,s}(\Omega)}^s.
\end{align}
Using a fundamental inequality, one gets that
\begin{align}\label{est:L-O-4}
\|\Phi(\mathbf{M}_{\alpha}\mathcal{G})\|_{L^{q,s}(\Omega)} & \le C \left( K_1 \sigma^{p_1} \varepsilon^{\frac{1}{q}} \|\Phi(\mathbf{M}_{\alpha}\mathcal{G})\|_{L^{q,s}(\Omega)} + K_1 \kappa^{p_1} \sigma^{p_1} \|\Phi(\mathbf{M}_{\alpha}\mathcal{F})\|_{L^{q,s}(\Omega)}\right).
\end{align}
Under the assumption of parameters as $\alpha \in [0,\frac{n}{\gamma})$, $0< s < \infty$ and $0 < q < \frac{n\gamma}{p_1(n - \gamma \alpha)}$, one may take $\varepsilon = \min\{\varepsilon_0,\varepsilon_1\}$ in~\eqref{est:L-O-4}, where $\varepsilon_1$ is chosen such that $ C K_1  \varepsilon_1^{\frac{1}{q}-p_1\left(\frac{1}{\gamma}-\frac{\alpha}{n}\right)} \le \frac{1}{2}$, to observe that
\begin{align}\label{est:L-O-5}
\|\Phi(\mathbf{M}_{\alpha}\mathcal{G})\|_{L^{q,s}(\Omega)}  & \le C_0  \|\Phi(\mathbf{M}_{\alpha}\mathcal{F})\|_{L^{q,s}(\Omega)},
\end{align}
where $C_0 = \max\left\{1, 2K_1 \kappa^{p_1} \sigma^{p_1}\right\}$. In a completely similar way we may prove~\eqref{est:L-O-5} even for the case $s = \infty$. By scaling on fractional-maximal distribution inequality~\eqref{est:O-1} with combining the convexity of $\Phi$, for all $\tau>0$ one also obtains that
\begin{align}\nonumber 
\left\|\Phi\left(\frac{\mathbf{M}_{\alpha}\mathcal{G}}{C_0\tau}\right)\right\|_{L^{q,s}(\Omega)} \le \frac{1}{C_0}\left\|\Phi\left(\frac{\mathbf{M}_{\alpha}\mathcal{G}}{\tau}\right)\right\|_{L^{q,s}(\Omega)}  & \le   \left\|\Phi\left(\frac{\mathbf{M}_{\alpha}\mathcal{G}}{\tau}\right)\right\|_{L^{q,s}(\Omega)}.
\end{align}
This inequality ensures that $\Lambda^{\mathcal{F}} \subset \frac{1}{C_0}\Lambda^{\mathcal{G}}$ which implies to~\eqref{eq:L-O}, where
\begin{align*}
\Lambda^f = \left\{\tau >0: \ \left\|\Phi\left(\frac{|f|}{\tau}\right)\right\|_{L^{q,s}(\Omega)} \le 1\right\},
\end{align*}

with $f = \mathbf{M}_{\alpha}\mathcal{G}$ or $f = \mathbf{M}_{\alpha}\mathcal{F}$.\\

Finally, under the local comparison~\ref{ing:A2_2}, the distribution inequality~\eqref{est:L-O-3} is valid even for constant $\sigma = \sigma_0$ which does not depend on $\varepsilon$. It follows that the exponent of $\varepsilon$ in~\eqref{est:L-O-4} is exactly $\frac{1}{q}$. For this reason, the estimate~\eqref{est:L-O-5} can be obtained with a relaxation on the constraints of $\alpha$ and $q$ in the previous case. Hence the inequality~\eqref{eq:L-O} holds for all $\alpha \in [0,n)$, $0< q < \infty$ and $0<s \le \infty$. The proof is complete.
\end{proof}

\section{Applications in regularity theory}\label{sec:app}
Based on the very effective technique with FMD, regularity estimates of solutions to a class of more general elliptic/parabolic equations will follow as an application. In this section, we apply the previous abstract results for investigating the regularity of weak solutions to both \emph{quasi-linear elliptic equations} and \emph{quasi-linear elliptic double obstacle problems} in prescribed spaces. For that purpose, we will utilize the theory of FMD method afore-discussed; and moreover, the well-proved norm estimates in Section~\ref{sec:abstract} will also be taken into account.
\subsection{Quasi-linear elliptic problems}
\subsubsection{Problem setting}
Let us consider the quasi-linear elliptic problems under nonhomogeneous Dirichlet boundary condition of the type
\begin{align}\tag{$\mathbf{P_1}$}
\label{eq:app}
\begin{cases}
\mathrm{div}(\mathbb{A}(x,\nabla u))  &= \ \mathrm{div}(\mathbb{B}(x,\mathbf{F})), \quad \ \text{in} \ \ \Omega, \\
\hspace{1.2cm} u &=\ \ \mathsf{g}, \qquad \hspace{1.3cm} \ \text{on} \ \ \partial \Omega.
\end{cases}
\end{align}
Here, $\Omega$ is an open bounded domain of $\mathbb{R}^n$ with suitable requirements for $\partial\Omega$ (will be discussed later); $\mathbf{F} \in L^p(\Omega;\mathbb{R}^n)$; boundary data $\mathsf{g} \in W^{1,p}(\Omega)$ for $p>1$. Further,  the quasi-linear operators $\mathbb{A}, \mathbb{B}: \Omega \times \mathbb{R}^n \rightarrow \mathbb{R}$ are Carath\'eodory vector valued functions (i.e, they are measurable respect to $\xi$ on $\Omega$ for every $\xi$ in $\mathbb{R}^n$ and they are continuous on $\mathbb{R}^n$ for almost every $x$ in $\Omega$) satisfying the $p$-monotone conditions ($p \in (1,n]$): there exist constants $\Lambda>0$ and $\varsigma \in [0,1]$ such that 
\begin{align}\label{eq:A1}
& \left| \mathbb{A}(x,\xi) \right|  \le \Lambda \left(\varsigma^2 + |\xi|^2 \right)^{\frac{p-1}{2}}, \\ 
\label{eq:A2}
& \langle \mathbb{A}(x,\xi_1)-\mathbb{A}(x,\xi_2), \xi_1 - \xi_2 \rangle \ge \Lambda^{-1} \Psi_{\varsigma}(\xi_1, \xi_2),\\ \label{eq:B}
& \left| \mathbb{B}(x,\xi) \right| \le \Lambda |\xi|^{p-1},
\end{align}
for almost every $x$ in $\Omega$ and every $\xi$, $\xi_1$, $\xi_2 \in \mathbb{R}^n \setminus \{0\}$, where the function $\Psi_{\varsigma}$ is defined by
\begin{align}\label{def:Phi}
\Psi_{\varsigma}(\xi_1, \xi_2) : = \left(\varsigma^2 + |\xi_1|^2 + |\xi_2|^2 \right)^{\frac{p-2}{2}}|\xi_1 - \xi_2|^2, \quad \xi_1, \, \xi_2 \in \mathbb{R}^n.
\end{align}

Classically, there exists a weak solution $u \in W^{1,p}(\Omega)$ to problems~\eqref{eq:app}, that is
\begin{align*}
\int_\Omega{\langle\mathbb{A}(x,\nabla u), \nabla\varphi \rangle dx} = \int_\Omega{\langle \mathbb{B}(x,\mathbf{F}), \nabla\varphi \rangle dx},
\end{align*}
holds for all $\varphi \in W_0^{1,p}(\Omega)$. For the sake of brevity, we use the notation $$u \in \mathrm{sol}(\mathbb{A},\mathbb{B}(\cdot,\mathbf{F}),\mathsf{g};\Omega)$$ to say that the function $u \in W^{1,p}(\Omega)$ is a weak solution to problem~\eqref{eq:app}.

The presence of $\varsigma$ in $\mathbb{A}$ brings a challenging feature in this kind of problems~\eqref{eq:app}. It is remarkable that for the special case when $\varsigma=0$ and $\mathbb{B}(x,\xi)=|\xi|^{p-2}\xi$, we have the certain quasi-linear divergence form elliptic equations, and regularity results here recover what obtained in our previous works in~\cite{PNJDE, MPTNsub, PNmix} in Lorentz, Lorentz-Morrey spaces and several research papers~\cite{BYZ2008} are matched in the Orlicz realm.  Otherwise with $s=1$ and $p=2$ see~\cite{MP11}, etc. Among the other works on this topic of investigation, we also must quote~\cite{Phuc2015,Baroni2013,BCDKS, SSB4, CoMi2016, FT2018, Mi3} and a huge literature in recent years, where given assumptions of $\mathbb{A, B}$, $\Omega$ are suitably changed. Besides, there have been a lot of works treating the study of regularity of nonlinear elliptic equations with $\varsigma=0$ and it would be impossible to report all progresses have been made here. For this reason, we have chosen some recent results to mention. In addition, for the case of non-degenerate condition when $\varsigma>0$, interested readers may see~\cite{Mi3,Duzamin2,55DuzaMing,KM2012,KM2014} - various works by Mingione \textit{et al.} during last several years. 

It is also worth mentioning that global regularity results are obtainable using this embedding technique with an extra assumptions on $\Omega$: $p$-capacity thick complement. To better illustrate the principal ideas in the proofs, this condition is essentially sharp for our domain. There is still interesting to extend our results with $\Omega$ is a Reifenberg flat domain. Not too far way from our main objective in this paper, we recommend~\cite{PNmix,PNJDE,MPT2018,BW2,MP11,SSB4}  to which the definitions of $p$-capacity of an arbitrary set, geometrical structure and properties of these types of domains, also for small BMO condition (BMO coefficients with small BMO semi-norms) can be found. 

Why two types of domains are independently considered in our work? There are many reasons here:
\begin{itemize}
\item[-] Distinction between those two types is actually only from the boundary regularity estimates.
\item[-] Under the assumption of $p$-capacity thick complement (see~\eqref{hyp:P}), local comparison~\ref{ing:A2_1} is in use. Otherwise, with assumption of Reifenberg flat domain (see~\eqref{hyp:R}) plus the small BMO condition (see~\eqref{cond:BMO}), one needs the ingredient~\ref{ing:A2_2}. As above-mentioned in Remark~\ref{rem:ingreA2}, these ingredients are separate considered regarding to different assumptions on $\Omega$.
\item[-] In fact, the assumption~\eqref{hyp:P} is essentially sharp for our global regularity results, but the presence of assumption~\eqref{hyp:R} and~\eqref{cond:BMO} is also considered to conclude the broader ranges for scales $q$ and $s$, see~\cite{PNmix}.
\item[-] Problem~\eqref{eq:app} going with hypothesis~\eqref{hyp:R} on $\Omega$ plus condition~\eqref{cond:BMO} confirms the higher regularity results than with only assumption~\eqref{hyp:P} on $\Omega$. They are comparable with each other, see Theorem~\ref{theo:app-1} and~\ref{theo:app-2} below for details.
\end{itemize}

\subsubsection{Comparison estimate}
This section is dedicated to the proofs of local comparison estimates (in the interior and on the boundary of the domain). As a matter of fact, we are going to prove ingredient~\ref{ing:A2_1} and~\ref{ing:A2_2} for our problem, respectively.

\begin{lemma}\label{lem:Phi(v,v)}
Let $B$ be open bounded subset of $\mathbb{R}^n$ and two functions $\phi_1$, $\phi_2 \in L^p(B)$ with $p>1$. Then for every $\varepsilon \in (0,1)$, there holds
\begin{align}\label{eq:lem-Phi(u,v)}
\fint_{B} |\phi_1 - \phi_2|^p dx \le \varepsilon  \fint_{B} \left(\varsigma^p + |\phi_1|^p\right) dx + \max\left\{1, 8\varepsilon^{1 - \frac{2}{p}}\right\} \fint_{B} \Psi_{\varsigma}(\phi_1, \phi_2) dx.
\end{align}
\end{lemma}
\begin{proof}
In the first case when $p \ge 2$, it is easily to see that $|\phi_1 - \phi_2|^p \le \Psi_{\varsigma}(\phi_1, \phi_2)$ which ensures~\eqref{eq:lem-Phi(u,v)} even for $\varepsilon=0$.  Otherwise, if $1< p < 2$, we first decompose $|\phi_1 - \phi_2|^p$ as follows
\begin{align}\nonumber
|\phi_1 - \phi_2|^p & = \left(\varsigma^2 + |\phi_1|^2 + |\phi_2|^2\right)^{\frac{p(2-p)}{4}} [ \Psi_{\varsigma}(\phi_1, \phi_2)]^{\frac{p}{2}} \\ \nonumber
&\le 2^{\frac{p(2-p)}{4}}\left(\varsigma^2 + |\phi_1|^2 + |\phi_1-\phi_2|^2\right)^{\frac{p(2-p)}{4}} [ \Psi_{\varsigma}(\phi_1, \phi_2)]^{\frac{p}{2}}\\ \label{est:123}
& \le 2^{\frac{p(2-p)}{4}}\left(\varsigma^p + |\phi_1|^p + |\phi_1-\phi_2|^p\right)^{1-\frac{p}{2}} [ \Psi_{\varsigma}(\phi_1, \phi_2)]^{\frac{p}{2}}.
\end{align}
Here the last estimate can be considered as an application of the following fundamental  inequality for $m \ge 1$ non-negative numbers $\alpha_1$, $\alpha_2$, ..., $\alpha_m$ given as
\begin{align}\label{inq:general}
\left(\sum_{i=1}^m \alpha_i \right)^{s} \le \max\left\{1; m^{s-1}\right\} \sum
_{i=1}^m \alpha_i^{s}, \quad \mbox{ for all } s \ge 0.
\end{align}
For all $\epsilon \in (0,1/2)$, we apply H{\"o}lder and Young inequalities on the right hand side of~\eqref{est:123} to get that
\begin{align}\nonumber
\fint_{B} |\phi_1 - \phi_2|^p dx & \le \left(\fint_{B} \left(\varsigma^p + |\phi_1|^p + |\phi_1-\phi_2|^p\right) dx\right)^{1-\frac{p}{2}} \left(2\fint_{B} \Psi_{\varsigma}(\phi_1, \phi_2) dx\right)^{\frac{p}{2}} \\ \nonumber
& \le \epsilon  \fint_{B} \left(\varsigma^p + |\phi_1|^p + |\phi_1-\phi_2|^p\right) dx + 2\epsilon^{1 - \frac{2}{p}} \fint_{B} \Psi_{\varsigma}(\phi_1, \phi_2) dx \\ \nonumber 
& \le \frac{1}{2} \fint_{B} |\phi_1 - \phi_2|^p dx +  \epsilon  \fint_{B} \left(\varsigma^p + |\phi_1|^p\right) dx + 2\epsilon^{1 - \frac{2}{p}} \fint_{B} \Psi_{\varsigma}(\phi_1, \phi_2) dx,
\end{align}
which implies to~\eqref{est:u-v-DOP} by replacing $\varepsilon = 2\epsilon$.
\end{proof}

\begin{lemma}\label{lem:u-v-PHI}
Let $B$ be an open subset of $\Omega$, the nonlinear operators $\mathbb{A}$ satisfies assumptions~\eqref{eq:A1}-\eqref{eq:A2} and $\mathbb{B}$ satisfies condition~\eqref{eq:B} for $p>1$. Assume that 
\begin{align*}
u \in \mathrm{sol}(\mathbb{A},\mathbb{B}(\cdot,\mathbf{F}),\mathsf{g};\Omega) \ \ \mbox{ and } \ \ v \in \mathrm{sol}(\mathbb{A},0,u-\mathsf{g};B),
\end{align*} 
for given data $\mathbf{F} \in L^p(\Omega)$ and $\mathsf{g} \in W^{1,p}(\Omega)$. Then for every $\varepsilon \in (0,1)$, one can find $C = C(p,\varepsilon)>0$ such that
\begin{align}\label{est:u-v-PHI}
\fint_{B} |\nabla u - \nabla v|^p dx \le \varepsilon  \fint_{B} \left(\varsigma^p + |\nabla u|^p\right) dx + C \fint_{B} \left(|\mathbf{F}|^{p} + |\nabla \mathsf{g}|^p\right) dx.
\end{align}
\end{lemma}
\begin{proof}
Let $u \in \mathrm{sol}(\mathbb{A},\mathbb{B}(\cdot,\mathbf{F}),\mathsf{g};\Omega)$ and $v \in \mathrm{sol}(\mathbb{A},0,u-\mathsf{g};B)$ for given data $\mathbf{F} \in L^p(\Omega)$ and $\mathsf{g} \in W^{1,p}(\Omega)$. Subtracting the variational formulas corresponding to weak solutions $u$ and $v$ respectively, one obtains that
\begin{align*}
\int_{B} \langle \mathbb{A}(x,\nabla u)- \mathbb{A}(x,\nabla v), \nabla \phi \rangle dx =  \int_{B} \langle \mathbb{B}(x,\mathbf{F}), \nabla \phi\rangle dx, 
\end{align*}
for every test function $\phi \in W^{1,p}_0(B)$. We take $u - v - \mathsf{g} \in W^{1,p}_0(B)$ as the test function in this formula to get that
\begin{align*}
\int_{B} \langle \mathbb{A}(x,\nabla u) - \mathbb{A}(x,\nabla v), \nabla u - \nabla v  \rangle dx & =  \int_{B} \langle \mathbb{A}(x,\nabla u), \nabla \mathsf{g} \rangle dx - \int_{B} \langle \mathbb{A}(x,\nabla v), \nabla \mathsf{g} \rangle dx \\ 
& \quad + \int_{B} \langle \mathbb{B}(x,\mathbf{F}), \nabla  u - \nabla v\rangle dx -  \int_{B} \langle \mathbb{B}(x,\mathbf{F}), \nabla \mathsf{g}\rangle dx.
\end{align*}
Applying assumptions~\eqref{eq:A1} and~\eqref{eq:A2}, it deduces that
\begin{align}\nonumber
\Lambda^{-1}\int_{B} \Psi_{\varsigma}(\nabla u, \nabla v) dx & \le \Lambda \int_{B} \left[ \left(\varsigma^2  + |\nabla u|^2\right)^{\frac{p-1}{2}}  + \left(\varsigma^2 + |\nabla v|^2\right)^{\frac{p-1}{2}}\right] |\nabla \mathsf{g}| dx \\ \label{est:app-5}
& \hspace{2cm} +   \int_{B} |\mathbf{F}|^{p-1} |\nabla u - \nabla v| dx + \int_{B} |\mathbf{F}|^{p-1} |\nabla \mathsf{g}| dx,
\end{align}
where the function $\Psi_{\varsigma}$ is defined as in~\eqref{def:Phi}. Moreover, thanks to inequality~\eqref{inq:general} one notes that
\begin{align*}
\left(\varsigma^2 + |\nabla v|^2\right)^{\frac{p-1}{2}}  \le \left(\varsigma + |\nabla u| + |\nabla u - \nabla v|\right)^{p-1} \le C(p) \left(\varsigma^{p-1} + |\nabla u|^{p-1} + |\nabla u - \nabla v|^{p-1}\right).
\end{align*}
It follows from~\eqref{est:app-5} that
\begin{align}\nonumber
\int_{B} \Psi_{\varsigma}(\nabla u, \nabla v) dx & \le C(\Lambda,p) \left[ \int_{B} \left(\varsigma^{p-1} + |\nabla u|^{p-1} + |\nabla u - \nabla v|^{p-1}\right) |\nabla \mathsf{g}| dx \right. \\ \label{est:app-6}
& \hspace{2cm} \left. +   \int_{B} |\mathbf{F}|^{p-1} |\nabla u - \nabla v| dx + \int_{B} |\mathbf{F}|^{p-1} |\nabla \mathsf{g}| dx\right].
\end{align}
Thanks to  H{\"o}lder and Young inequalities with every $\varepsilon_1$, $\varepsilon_2>0$ for all terms on the right hand side of~\eqref{est:app-6}, one gets that
\begin{align}\nonumber
\fint_{B} \Psi_{\varsigma}(\nabla u, \nabla v) dx & \le  \varepsilon_1 \fint_{B} |\nabla u - \nabla v|^{p}  dx  + \varepsilon_2 \fint_{B} \left(\varsigma^{p} + |\nabla u|^{p}\right)  dx  \\ \label{est:app-7}
& \hspace{3cm} +  C(\Lambda,p,\varepsilon_1,\varepsilon_2)  \fint_{B} \left( |\mathbf{F}|^{p} + |\nabla \mathsf{g}|^p \right) dx.
\end{align}
Applying Lemma~\ref{lem:Phi(v,v)}, it allows us to get that
\begin{align}\label{est:app-8}
\fint_{B} |\nabla u - \nabla v|^p dx  \le \varepsilon_3  \fint_{B} \left(\varsigma^p + |\nabla u|^p\right) dx + C(p,\varepsilon_3) \fint_{B} \Psi_{\varsigma}(\nabla u, \nabla v) dx,
\end{align}
for all $\varepsilon_3>0$. Substituting estimate~\eqref{est:app-7} into~\eqref{est:app-8}, one has
\begin{align*}
\fint_{B} |\nabla u - \nabla v|^p dx & \le  \varepsilon_1  C(p,\varepsilon_3)   \fint_{B} |\nabla u-\nabla v|^p dx \\
& \hspace{2cm} +  \left(\varepsilon_3 + \varepsilon_2 C(p,\varepsilon_3)\right) \fint_{B} \left(\varsigma^p + |\nabla u|^p\right) dx   \\
& \hspace{4cm} + C(\Lambda,p,\varepsilon_1,\varepsilon_2,\varepsilon_3)  \fint_{B} \left( |\mathbf{F}|^{p} + |\nabla \mathsf{g}|^p \right) dx,
\end{align*}
which allows us to conclude~\eqref{est:u-v-PHI} by taking $\varepsilon_1 = \frac{1}{2} \left(C(p,\varepsilon_3)\right)^{-1}$, $\varepsilon_2 =  \varepsilon_3 \left(C(p,\varepsilon_3)\right)^{-1}$ and then replacing $4\varepsilon_3$ by $\varepsilon \in (0,1)$.
\end{proof}

\subsubsection{Global regularity results}\label{sec:L_estimates}
For the reader's convenience, we provide here the additional definition of \emph{domain with $p$-capacity uniform thick complement}.
\begin{definition}[Domain with $p$-capacity uniform thick complement]\label{def:pcapa}
The complement domain of $\Omega$ in $\mathbb{R}^n$, $\mathbb{R}^n \setminus \Omega$ is said to satisfy the $p$-capacity uniform thickness condition if there exist two constants $c_0,r_0>0$ such that
\begin{align}\tag{$\mathcal{HP}$}\label{hyp:P}
\mathrm{cap}_p((\mathbb{R}^n \setminus \Omega) \cap \overline{B}_{\varrho}(\xi); B_{2\varrho}(\xi)) \ge c_0 \mathrm{cap}_p(\overline{B}_{\varrho}(\xi); B_{2\varrho}(\xi)),
\end{align}
for any $0<\varrho \le r_0$ and $\xi \in \mathbb{R}^n \setminus \Omega$. 
\end{definition}
Here, the $p$-capacity of $E \subset \mathbb{R}^n$, namely $\mathrm{cap}_p(E; \Omega)$, will be defined as follows
\begin{align*}
\mathrm{cap}_p(E;\Omega) = \inf_{E_1 \text{ open}, \ E_1 \subseteq E} \left\{ \sup_{E_2  \text{ compact},\ E_2 \subseteq E_1} \left( \inf_{\phi \in C_c^\infty(\Omega), \ \chi_{E_2}\phi \ge 1}   \int_\Omega{|\nabla \phi|^p dx} \right) \right\}.
\end{align*}

\begin{theorem}[Global Lorentz estimates under assumption~\eqref{hyp:P}]
\label{theo:app-1}
Let $\Omega \subset \mathbb{R}^n$ be an open bounded domain satisfying~\eqref{hyp:P} with $c_0, r_0>0$. Assume that operator $\mathbb{A}$ satisfies assumptions~\eqref{eq:A1}-\eqref{eq:A2} and $\mathbb{B}$ satisfies condition~\eqref{eq:B}. Given data $\mathbf{F} \in L^p(\Omega)$ and $\mathsf{g} \in W^{1,p}(\Omega)$ for $p \in (1,n]$, suppose that 
\begin{align*}
u \in \mathrm{sol}(\mathbb{A},\mathbb{B}(\cdot,\mathbf{F}),\mathsf{g};\Omega).
\end{align*}
Then there exists $\Theta>p$ such that if $\mathbf{M}_{\alpha}(|\mathbf{F}|^p + |\nabla \mathsf{g}|^p) \in L^{q,s}(\Omega)$ for $0 \le \alpha < \frac{n p}{\Theta}$, $0<q<\frac{n \Theta}{np - \alpha \Theta}$ and $0 < s \le \infty$ then $\mathbf{M}_{\alpha}((\varsigma +|\nabla u|)^p) \in L^{q,s}(\Omega)$ corresponding to the following estimate
\begin{align}\label{est:theo-app-1}
\|\mathbf{M}_{\alpha}((\varsigma +|\nabla u|)^p)\|_{L^{q,s}(\Omega)} \le C\|\mathbf{M}_{\alpha}(|\mathbf{F}|^p + |\nabla \mathsf{g}|^p)\|_{L^{q,s}(\Omega)}.
\end{align}
\end{theorem}
\begin{proof}
The main idea of this proof is to apply Theorem~\ref{theo:norm-L-1} for two functions defined by
\begin{align}\label{est:app-FG}
\mathcal{F} = |\mathbf{F}|^p + |\nabla \mathsf{g}|^p\ \mbox{ and }\ \mathcal{G} = (\varsigma +|\nabla u|)^p.
\end{align}
It is sufficient to show that all hypotheses of Theorem~\ref{theo:G-lam-A} are valid. Firstly, the global comparison~\ref{ing:A3} holds by applying Lemma~\ref{lem:u-v-PHI} for $B = \Omega$ with noting that $v \equiv 0$ in this case. 

Next, for every $0< r \le r_0/2$ and $\nu \in \overline{\Omega}$, let us consider two functions $\varphi = (\varsigma + |\nabla v|)^p$ and $\psi = |\nabla u - \nabla v|^p$, where  $v \in \mathrm{sol}(\mathbb{A},0,u-\mathsf{g};\Omega_{2r}(\nu))$. A trivial verification shows that $(\mathcal{G},\varphi,\psi)$ satisfies a quasi-triangle condition~\eqref{cond:phi-psi} with $\tilde{c} = 3^{p-1}$. Moreover, we refer the reader to~\cite[Theorem 10]{Mi3} in which the authors proved that there exists a constant $\Theta = \Theta(n,p,\sigma,\Lambda)>p$ such that 
\begin{align}\label{est:app-RH}
\left(\fint_{\Omega_{r}(\nu)}(\varsigma+|\nabla v|)^{\Theta} dx\right)^{\frac{1}{\Theta}}\leq C\left(\fint_{\Omega_{2r}(\nu)}(\varsigma+|\nabla v|)^p dx\right)^{\frac{1}{p}}.
\end{align}
On the other hand, thanks to Lemma~\ref{lem:u-v-PHI} with $B = \Omega_r(\nu)$, there holds
\begin{align}\label{est:app-u-v}
\fint_{\Omega_{2r}(\nu)}{|\nabla u - \nabla v|^p dx} & \le \varepsilon   \fint_{\Omega_{2r}(\nu)}{(\varsigma +|\nabla u|)^p dx}  + C(\varepsilon,p)   \fint_{\Omega_{2r}(\nu)}{|\mathbf{F}|^p + |\nabla \mathsf{g}|^p dx}.
\end{align}
Two inequalities~\eqref{est:app-RH} and~\eqref{est:app-u-v} ensure that $\varphi \in \mathcal{RH}^{{\Theta}/{p}}(\Omega)$ and two functions $\mathcal{G}$, $\mathcal{F}$ satisfy the local comparison estimate~\ref{ing:A2_1}. Finally the proof is completed as an application of Theorem~\ref{theo:norm-L-1}.
\end{proof}

In the next theorem, we prove that~\eqref{est:theo-app-1} even holds for a larger range of parameters $\alpha$ and $q$ under a better assumption on the boundary of $\Omega$ and a mild hypothesis on the derivative of nonlinear operator $\mathbb{A}$, as follows
\begin{align}\label{eq:A1-b}
&  |D_{\xi} \mathbb{A}(x,\xi)| \le \Lambda \left(\varsigma^2 + |\xi|^2 \right)^{\frac{p-2}{2}},
\end{align}
for almost everywhere $x \in \Omega$ and $\xi \in \mathbb{R}^n$. We also recall here the definition of \emph{Reifenberg flat domain} below.

\begin{definition}[$(\delta,r_0)$-Reifenberg flat domain]\label{def:Reifenberg}
Let $\delta \in (0,1)$ and $r_0>0$, we say that $\Omega$ is a $(\delta,r_0)$-Reifenberg flat domain if for each $\xi_0 \in \partial \Omega$ and each $\varrho \in (0,r_0]$, one can find a coordinate system $\{z_1,z_2,...,z_n\}$ with origin at $\xi_0$ such that
\begin{align}\tag{$\mathcal{HR}$}\label{hyp:R}
B_{\varrho}(\xi_0) \cap \{z: \ z_n > \delta \varrho\} \subset B_{\varrho}(\xi_0) \cap \Omega \subset B_{\varrho}(\xi_0) \cap \{z: \ z_n > -\delta \varrho\},
\end{align}
where for simplicity, the set $\{z = (z_1, z_2, ..., z_n): \ z_n > k\}$ is denoted by $\{z: \ z_n > k\}$.
\end{definition}
\begin{definition}[$(\delta,r_0)$-BMO condition]\label{def:BMOcond}
The nonlinearity $\mathbb{A}$ is said to satisfy a $(\delta,r_0)$-BMO condition with exponent $t>0$ if the following condition holds
\begin{align}\tag{$\mathcal{HS}$}
\label{cond:BMO}
[\mathbb{A}]_t^{r_0} = \sup_{y \in \mathbb{R}^n, \ 0<\varrho\le r_0} \left(\fint_{B_{\varrho}(y)} \left(\sup_{z \in \mathbb{R}^n \setminus \{0\}} \frac{|\mathbb{A}(x,z) - \overline{\mathbb{A}}_{B_{\varrho}(y)}(z)|}{|z|^{p-1}}\right)^t dx\right)^{\frac{1}{t}} \le \delta,
\end{align}
where $\overline{\mathbb{A}}_{B_{\varrho}(y)}(z)$ denotes the average of $\mathbb{A}(\cdot,z)$ over the ball $B_{\varrho}(y)$.
\end{definition}
\begin{theorem}[Global Lorentz estimates under assumption \eqref{hyp:R}]
\label{theo:app-2}
Assume that operator $\mathbb{A}$ satisfies assumptions~\eqref{eq:A1}-\eqref{eq:A2}-\eqref{eq:A1-b} and $\mathbb{B}$ satisfies condition~\eqref{eq:B}. Let $u \in \mathrm{sol}(\mathbb{A},\mathbb{B}(\cdot,\mathbf{F}),\mathsf{g};\Omega)$ with given data $\mathbf{F} \in L^p(\Omega)$ and $\mathsf{g} \in W^{1,p}(\Omega)$ for $p \in (1,n]$. Then there exist $\delta_0>0$ and  $p_0>p$ such that if $\Omega$ is a $(r_0,\delta_0)$-Reifenberg flat domain for some $r_0>0$ and $[\mathbb{A}]_{p_0}^{r_0} \le \delta_0$ then~\eqref{est:theo-app-1} holds for all $\alpha \in [0,n)$, $q \in (0,\infty)$ and $0 < s \le \infty$.  
\end{theorem}
\begin{proof}
Similar to the proof of Theorem~\ref{theo:app-1}, inequality~\eqref{est:theo-app-1} can be observed by applying Theorem~\ref{theo:norm-L-2}. Let us consider $\mathcal{F}$ and $\mathcal{G}$ defined by the same formula as in~\eqref{est:app-FG}. For every $\nu \in \overline{\Omega}$ and $0< r \le r_0/2$ with given $r_0>0$, we define $v \in \mathrm{sol}(\mathbb{A},0,u-\mathsf{g};\Omega_{2r}(\nu))$ and $w \in \mathrm{sol}(\overline{\mathbb{A}}_{\Omega_{3r/2}(\nu)},0,v;\Omega_{3r/2}(\nu))$. Then the existence of two functions $\varphi$ and $\psi$ in Theorem~\ref{theo:G-lam-B} is given as follows
\begin{align*}
\varphi = (\varsigma + |\nabla w|)^p \ \mbox{ and } \ \psi = |\nabla u - \nabla w|^p.
\end{align*}
The same technique in~\cite{MP11} remains valid for our problem, which shows that
\begin{equation}\label{eq:lem2b}
\fint_{\Omega_{2r}(\nu)} (\varsigma + |\nabla w|)^p dx \le C\fint_{\Omega_{2r}(\nu)} (\varsigma + |\nabla v|)^p dx,
\end{equation}
and
\begin{equation}\label{eq:lem2a}
\fint_{\Omega_{r}(\nu)} |\nabla v - \nabla w|^p dx \le C \left([\mathbb{A}]^{r_0}_{p_0}\right)^p \fint_{\Omega_{2r}(\nu)} (\varsigma + |\nabla v|)^p dx,
\end{equation}
where $p_0 = \frac{p \Theta}{\Theta -p}$. Thanks to~\cite[Lemma 5]{Lieberman1988} and~\eqref{eq:lem2b}, there holds
\begin{align*}
\|\varsigma + |\nabla w|\|^p_{L^{\infty}(\Omega_{r}(\nu))} &\le C \fint_{\Omega_{2r}(\nu)} (\varsigma +|\nabla w|)^p dx   \le C \fint_{\Omega_{2r}(\nu)} (\varsigma + |\nabla v|)^p dx,
\end{align*}
which deduces from~\eqref{est:app-u-v} that
\begin{align}\label{est:app-w-inf}
\|\varsigma + |\nabla w|\|^p_{L^{\infty}(\Omega_{r}(\nu))} & \le C \left( \fint_{\Omega_{2r}(\nu)} (\varsigma +|\nabla u|)^p  dx +  \fint_{\Omega_{2r}(\nu)} (|\mathbf{F}|^p + |\nabla \mathsf{g}|^p) dx \right).
\end{align}
In addition, from~\eqref{eq:lem2a} one observes that
\begin{align*}
\fint_{\Omega_{r}(\nu)} |\nabla u - \nabla w|^p dx &\le C\left( \fint_{\Omega_{r}(\nu)} |\nabla u - \nabla v|^p dx +  \fint_{\Omega_{r}(\nu)} |\nabla v - \nabla w|^p dx \right) \\
& \le C\left( \fint_{\Omega_{2r}(\nu)} |\nabla u - \nabla v|^p dx + \left([\mathbb{A}]_{p_0}^{r_0}\right)^p \fint_{\Omega_{2r}(\nu)} (\varsigma + |\nabla v|)^p dx \right),
\end{align*}
which can be rewritten as below if $[\mathbb{A}]_{p_0}^{r_0}$ is small enough
\begin{align*}
\fint_{\Omega_{r}(\nu)} |\nabla u - \nabla w|^p dx & \le C\left( \fint_{\Omega_{2r}(\nu)} |\nabla u - \nabla v|^p dx + \left([\mathbb{A}]_{p_0}^{r_0}\right)^p \fint_{\Omega_{2r}(\nu)} (\varsigma +|\nabla u|)^p dx \right).
\end{align*}
Applying inequality~\eqref{est:app-u-v} with $\varepsilon = \left([\mathbb{A}]_{p_0}^{r_0}\right)^p$, it follows from above estimate that 
\begin{align}\label{est:uw}
\fint_{\Omega_{r}(\nu)} |\nabla u - \nabla w|^p dx & \le \tilde{C}\left(\left([\mathbb{A}]_{p_0}^{r_0}\right)^p \fint_{\Omega_{2r}(\nu)} (\varsigma +|\nabla u|)^p dx  + C\fint_{\Omega_{2r}(\nu)} (|\mathbf{F}|^p + |\nabla \mathsf{g}|^p)  dx \right),
\end{align}
where $\tilde{C}$ depends only on $p$ and $C$ depends on $p$, $[\mathbb{A}]_{p_0}^{r_0}$. Finally for every $\varepsilon \in (0,1)$, if $[\mathbb{A}]_{p_0}^{r_0} \le (\varepsilon \tilde{C}^{-1})^{\frac{1}{p}}$, inequality~\eqref{est:uw} implies that
\begin{align}\label{est:app-uw}
\fint_{\Omega_{r}(\nu)} |\nabla u - \nabla w|^p dx & \le \varepsilon \fint_{\Omega_{2r}(\nu)} (\varsigma +|\nabla u|)^p dx  + C(p,\varepsilon)\fint_{\Omega_{2r}(\nu)} (|\mathbf{F}|^p + |\nabla \mathsf{g}|^p) dx.
\end{align}
Two estimates in~\eqref{est:app-w-inf} and~\eqref{est:app-uw} show that the functions $\mathcal{G}$ and $\mathcal{F}$ satisfy local comparison~\ref{ing:A2_2}. The conclusion in Theorem~\ref{theo:norm-L-2} ensures the existence of constant $\delta_0>0$ such that~\eqref{est:theo-app-1} holds for every $\alpha \in [0,n)$, $q \in (0,\infty)$ and $0 < s \le \infty$ if provided $[\mathbb{A}]_{p_0}^{r_0} \le \delta_0$. We note that $\delta_0 = (\varepsilon \tilde{C}^{-1})^{\frac{1}{p}}$ corresponding to $\varepsilon$ fixed at the end of Theorem~\ref{theo:norm-L-2}.
\end{proof}

The next theorem states the regularity results in the setting of Orlicz-Lorentz spaces under both non-smooth assumptions on the boundary of the domain. As a special case of Orlicz-Lorentz, the global estimates in Orlicz setting will follow analogously.

\begin{theorem}[Global Orlicz-Lorentz estimates]
\label{theo:app-3}
Let Young function $\Phi \in \Delta_2$ and $p_1>1$ given as in Lemma~\ref{lem:Delta2_pro2}. Under hypotheses of Theorem~\ref{theo:app-1}, one can find $\gamma>1$ such that for every $\alpha \in [0,\frac{n}{\gamma})$, $0 < q < \frac{n\gamma}{p_1(n - \gamma \alpha)}$ and $0< s \le \infty$, if $\mathbf{M}_{\alpha}(|\mathbf{F}|^p + |\nabla \mathsf{g}|^p) \in L^{\Phi}(q,s)(\Omega)$ then $\mathbf{M}_{\alpha}((\varsigma +|\nabla u|)^p) \in L^{\Phi}(q,s)(\Omega)$ corresponding the following estimate
\begin{align}\label{eq:app-3-LO}
\|\mathbf{M}_{\alpha}((\varsigma +|\nabla u|)^p)\|_{L^{\Phi}(q,s)(\Omega)} \le C \|\mathbf{M}_{\alpha}(|\mathbf{F}|^p + |\nabla \mathsf{g}|^p)\|_{L^{\Phi}(q,s)(\Omega)}.
\end{align}
Moreover, under hypotheses of Theorem~\ref{theo:app-2}, there exist $\delta_0>0$ and  $p_0>p$ such that if $\Omega$ is a $(r_0,\delta_0)$-Reifenberg flat domain for some $r_0>0$ and $[\mathbb{A}]_{p_0}^{r_0} \le \delta_0$ then the inequality~\eqref{eq:app-3-LO} holds for all $\alpha \in [0,n)$, $0< q < \infty$ and $0<s \le \infty$.
\end{theorem}
\begin{proof}
As in the proof of two previous theorems, the functions $\mathcal{F}$ and $\mathcal{G}$ given as in~\eqref{est:app-FG} satisfy either local comparison~\ref{ing:A2_1} or~\ref{ing:A2_2} under hypotheses of Theorem~\ref{theo:app-1} or hypotheses of Theorem~\ref{theo:app-2} respectively. It allows us to apply Theorem~\ref{theo:L-O} to conclude~\eqref{eq:app-3-LO} in both cases. We remark that in the first case the constant $\gamma = \frac{\Theta}{p}$ which appears in the reverse H{\"o}lder inequality~\eqref{est:app-RH}.
\end{proof}

\subsection{Quasi-linear elliptic double obstacle problems}
\label{sec:obs_prob}
This section models with double obstacle problems, arising in many physical phenomena and other applications. The study of regularity theory of such problem (about the sharp integrability between gradient of solutions and of the obstacles) become a central and promising topic in recent years. There has been a great deal of works concerning this variational inequality problem. 

For one-sided obstacle problems, $C^{0,\alpha}$ and $C^{1,\alpha}$ regularity was done by Choe in \cite{Choe1991}, H\"older continuity by Eleuteri in \cite{Eleuteri2007}, Calder\'on-Zygmund type estimates in \cite{EH2008,BCW2012,BDM2011}, etc. Later, special attention has been driven to the obstacle problems with double constraints, so far a list of references can be found in \cite{MZ1986, MMV1989,BR2020} and many others.

In this section, the regularity results for the quasi-linear elliptic double obstacle problems will be considered as the next application of FMD method in our paper. 
\subsubsection{Problem setting}
Let us formulate the form of quasi-linear elliptic double obstacle problems as follows. 

Let $\Omega$ be an open bounded domain of $\mathbb{R}^n$, $p \in (1,\infty)$ and $\mathbf{F} \in L^p(\Omega;\mathbb{R}^n)$. Given $f_1$, $f_2 \in W^{1,p}(\Omega)$ are two fixed functions such that $f_1 \le f_2$ almost everywhere in $\Omega$ and $f_1 \le 0 \le f_2$ on $\partial \Omega$. We introduce the following convex admissible set related to $f_1$ and $f_2$ by
\begin{align}\label{def:S_0}
\mathcal{S}_0 = \left\{f \in W^{1,p}(\Omega): \ f_1 \le f \le f_2 \ \mbox{ a.e. in }\Omega\right\}.
\end{align}
The double obstacle problem is to find a weak solution $u \in \mathcal{S}_0$ satisfying the variational inequality 
\begin{align}\tag{$\mathbf{P_2}$}\label{eq:DOP}
\int_{\Omega} \langle \mathbb{A}(x,\nabla u), \nabla (u - \phi) \rangle dx \le \int_{\Omega} \langle \mathbb{B}(x, \mathbf{F}), \nabla (u - \phi) \rangle dx, 
\end{align}
for all $\phi \in \mathcal{S}_0$. Here, the quasi-linear operator $\mathbb{A}: \Omega \times \mathbb{R}^n \rightarrow \mathbb{R}$ is a vector valued function such that $\mathbb{A}(\cdot,\xi)$ is measurable on $\Omega$ for every $\xi$ in $\mathbb{R}^n$, $\mathbb{A}(x,\cdot)$ is differentiable on $\mathbb{R}^n$ for almost every $x$ in $\Omega$ and satisfying the following conditions: there exist constants $p \in (1,\infty)$, $\Lambda>0$ and $\varsigma \in [0,1]$ such that 
\begin{align}\label{eq:A1-DOP}
& \left| \mathbb{A}(x,\xi) \right|  \le \Lambda \left(\varsigma^2 + |\xi|^2 \right)^{\frac{p-1}{2}}, \\
\label{eq:A2-DOP}
& \langle \mathbb{A}(x,\xi_1)-\mathbb{A}(x,\xi_2), \xi_1 - \xi_2 \rangle \ge \Lambda^{-1} \Psi_{\varsigma}(\xi_1, \xi_2),
\end{align}
for almost every $x$ in $\Omega$ and every $\xi$, $\xi_1$, $\xi_2 \in \mathbb{R}^n \setminus \{0\}$, where the function $\Psi_{\varsigma}$ is defined by~\eqref{def:Phi}. The nonlinear operator $\mathbb{B}: \Omega \times \mathbb{R}^n \rightarrow \mathbb{R}$ is a vector valued function such that
\begin{align}\label{eq:B-DOP}
& \left| \mathbb{B}(x,\xi) \right| \le \Lambda (\varsigma^2 + |\xi|^2)^{\frac{p-1}{2}}.
\end{align}

We refer to~\cite{BR2020} for the existence and uniqueness of weak solution $u \in \mathcal{S}_0$ to the double obstacle problem under the monotone and coercive assumptions of operator $\mathbb{A}$, with the following estimate
\begin{align*}
\|\nabla u\|_{L^p(\Omega)} \le C \left(\|\mathbf{F}\|_{L^p(\Omega)} + \|\nabla f_1\|_{L^p(\Omega)} + \|\nabla f_2\|_{L^p(\Omega)}\right).
\end{align*}

In this paper, as another application of FMD theory, we extends regularity  results for double obstacle problem~\eqref{eq:DOP} in Lorentz and Orlicz spaces.
 
\subsubsection{Comparison estimate}
For simplicity of  notation, we denote by
$u \in \mathrm{sol}^{\mathbf{DOP}}\left(\mathbb{A},\mathbb{B}(\cdot,\mathbf{F}),\mathcal{S}_0;\Omega\right)
$ to say that the function $u \in \mathcal{S}_0$ is a weak solution to problem~\eqref{eq:DOP}. Let us  recall the weak maximum principle already stated in~\cite[Lemma 3.5]{BCW2012}.
\begin{lemma}\label{lem:weak-max}
Assume that $w_1, w_2 \in W^{1,p}(\Omega)$ for some $p>1$ satisfying $(w_1 - w_2)^+ \in W_0^{1,p}(\Omega)$, $w_1 \le w_2$ on $\partial \Omega$ and the following variational formula
\begin{align*}
\int_{\Omega} \left\langle\mathbb{A}(x,\nabla w_1), \nabla \phi\right\rangle dx \le \int_{\Omega} \left\langle\mathbb{A}(x,\nabla w_2), \nabla \phi\right\rangle dx,  
\end{align*}
holds for all non-negative $\phi \in W_0^{1,p}(\Omega)$. Then $w_1 \le w_2$ almost everywhere in $\Omega$.
\end{lemma}
\begin{lemma}\label{lem:u-v-DOP}
Let $B$ be an open subset of $\Omega$, the operators $\mathbb{A}$, $\mathbb{B}$ satisfy assumptions in~\eqref{eq:A1-DOP}-\eqref{eq:A2-DOP} and~\eqref{eq:B-DOP} for $p>1$. Given functions $\mathbf{F} \in L^p(\Omega)$, $f_1$, $f_2 \in W^{1,p}(\Omega)$ such that $f_1 \le f_2$ almost everywhere in $\Omega$ and $f_1 \le 0 \le f_2$ on $\partial \Omega$. Assume that 
\begin{align*}
u \in \mathrm{sol}^{\mathbf{DOP}}\left(\mathbb{A},\mathbb{B}(\cdot,\mathbf{F}),\mathcal{S}_0;\Omega\right) \ \mbox{ and } \ v \in \mathrm{sol}\left(\mathbb{A},0,u;B\right),
\end{align*}
with $\mathcal{S}_0$ defined as in~\eqref{def:S_0}. Then for every $\varepsilon \in (0,1)$, one can find $C = C(p,\varepsilon)>0$ such that
\begin{align}\label{est:u-v-DOP}
\fint_{B} |\nabla u - \nabla v|^p dx \le \varepsilon  \fint_{B} \left(\varsigma^p + |\nabla u|^p\right) dx + C \fint_{B} \left(\varsigma^p + |\mathbf{F}|^{p} + |\nabla f_1|^p + |\nabla f_2|^p\right) dx.
\end{align}
\end{lemma}
\begin{proof}
We first introduce the following set
\begin{align*}
\mathcal{R} = \left\{\omega \in u + W_0^{1,p}(B): \ \omega \ge f_1 \ \mbox{ a.e. in } \ B\right\}.
\end{align*}
Let ${\omega}_1 \in \mathcal{R}$ be the unique solution to the variational inequality
\begin{align}\label{est:DOP-1}
\int_{B} \langle \mathbb{A}(x,\nabla {\omega}_1), \nabla {\omega}_1 - \nabla \phi \rangle dx \le \int_{B} \langle \mathbb{A}(x,\nabla f_2), \nabla {\omega}_1 - \nabla \phi \rangle dx,
\end{align}
for all test function $\phi \in \mathcal{R}$. By taking $u$ as the test function in~\eqref{est:DOP-1}, one has
\begin{align*}
\int_{B} \langle \mathbb{A}(x,\nabla {\omega}_1), \nabla {\omega}_1 \rangle dx \le \int_{B} \langle \mathbb{A}(x,\nabla {\omega}_1), \nabla u \rangle dx + \int_{B} \langle \mathbb{A}(x,\nabla f_2), \nabla {\omega}_1 - \nabla u \rangle dx,
\end{align*}
which with~\eqref{eq:A1-DOP} and~\eqref{eq:A2-DOP} implies to
\begin{align*}
\Lambda^{-1} & \int_{B} \left(\varsigma^2 + |\nabla {\omega}_1|^2\right)^{\frac{p-2}{2}}|\nabla {\omega}_1|^2 dx  \le \Lambda \left(\int_{B} (\varsigma^2 + |\nabla {\omega}_1|^2)^{\frac{p-1}{2}} |\nabla u| dx  \right.  \\   
& \hspace{3cm} \left. + \int_{B} (\varsigma^2 +|\nabla f_2|^2)^{\frac{p-1}{2}} |\nabla {\omega}_1| dx + \int_{B} (\varsigma^2 +|\nabla f_2|^2)^{\frac{p-1}{2}} |\nabla u| dx  \right).
\end{align*}
Proceeding as for the proof of Lemma~\ref{lem:u-v-PHI} with this inequality, it is not difficult to show that
\begin{align}\label{est:DOP-2}
\int_{B} |\nabla {\omega}_1|^p dx \le C \int_{B} \left(\varsigma^p + |\nabla f_2|^p + |\nabla u|^p \right) dx.
\end{align}
We now consider ${\omega}_1 - ({\omega}_1-f_2)^{+}$ as the test function in~\eqref{est:DOP-1}, there holds
\begin{align*}
\int_{B} \left \langle \mathbb{A}(x,\nabla {\omega}_1) - \mathbb{A}(x,\nabla f_2), \nabla \left(({\omega}_1 - f_2)^{+}\right) \right \rangle dx \le 0,
\end{align*}
which with~\eqref{eq:A2-DOP} implies to
\begin{align}\label{est:DOP-3}
\int_{B\cap \mathcal{E}} \Psi_{\varsigma}(\nabla {\omega}_1, \nabla f_2) dx \le 0, \quad \mbox{ where } \ \mathcal{E} = \{{\omega}_1 \ge f_2\}.
\end{align}
Moreover, it is known that for every $\varepsilon>0$, one has
\begin{align}\nonumber
\int_{B\cap \mathcal{E}} |\nabla({\omega}_1- f_2)|^p dx & \le \varepsilon    \int_{B\cap \mathcal{E}} \left(\varsigma^p + |\nabla {\omega}_1|^p + |\nabla f_2|^p \right) dx \\ \nonumber
& \hspace{3cm} + c_{\varepsilon} \int_{B\cap \mathcal{E}} \Psi_{\varsigma}(\nabla {\omega}_1, \nabla f_2) dx \\ \nonumber
& \le \varepsilon \int_{B\cap \mathcal{E}} \left(\varsigma^p + |\nabla {\omega}_1|^p + |\nabla f_2|^p \right) dx,
\end{align}
where the last inequality is valid from~\eqref{est:DOP-3}. This inequality can be rewritten as
\begin{align}\label{est:DOP-4}
\int_{B} |\nabla(({\omega}_1- f_2)^+)|^p dx  & \le \varepsilon \int_{B\cap \mathcal{E}} \left(\varsigma^p + |\nabla {\omega}_1|^p + |\nabla f_2|^p \right) dx,
\end{align}
Passing $\varepsilon$ to $0$ in~\eqref{est:DOP-4}, one concludes that ${\omega}_1 \le f_2$ almost everywhere in $B$. Hence ${\omega}_1 - u \in W_0^{1,p}(B)$ and $f_1 \le {\omega}_1 \le f_2$ almost everywhere in $B$. For this reason, we may extend ${\omega}_1$ to $\Omega \setminus B$ by $u$ so that ${\omega}_1 \in \mathcal{S}_0$ and ${\omega}_1-u = 0$ in $\Omega \setminus B$. Adding two inequality corresponding to the ones by taking ${\omega}_1$ and $u$ as test functions of problems~\eqref{eq:DOP} and~\eqref{est:DOP-1} respectively, one has
\begin{align}\label{est:DOP-5}
\int_{B} \left\langle \mathbb{A}(x,\nabla u) - \mathbb{A}(x,\nabla {\omega}_1), \nabla u - \nabla {\omega}_1 \right\rangle dx \le \int_{B} \left\langle \mathbb{B}(x,\mathbf{F}) - \mathbb{A}(x,\nabla f_2), \nabla u - \nabla {\omega}_1 \right\rangle dx.
\end{align}
Combining~\eqref{est:DOP-5} with assumptions~\eqref{eq:A1-DOP}, \eqref{eq:A2-DOP} and~\eqref{eq:B-DOP} on nonlinear operators $\mathbb{A}$, $\mathbb{B}$, it follows that
\begin{align}\nonumber
 \int_{B} \Psi_{\varsigma}(\nabla u,\nabla {\omega}_1) dx & \le \Lambda^2 \left( \int_{B} \left(\varsigma^2+|\mathbf{F}|^2\right)^{\frac{p-1}{2}} |\nabla u - \nabla {\omega}_1| dx \right. \\ \label{est:DOP-6}
  & \hspace{2cm} \left. + \int_{B} \left(\varsigma^2 + |\nabla f_2|^2\right)^{\frac{p-1}{2}} |\nabla u - \nabla {\omega}_1| dx\right).
\end{align}
Applying H{\"o}lder's and Young's inequalities for every $\varepsilon_1 >0$ for two terms on the right hand side of~\eqref{est:DOP-6}, we gets that
\begin{align}\label{est:DOP-7}
\int_{B} \Psi_{\varsigma}(\nabla u,\nabla {\omega}_1) dx & \le \varepsilon_1 \int_{B} |\nabla u - \nabla {\omega}_1|^p dx + C(p,\varepsilon_1) \int_{B} \left(\varsigma^{p} + |\mathbf{F}|^p + |\nabla f_2|^p\right) dx.
\end{align}
Thanks to Lemma~\ref{lem:Phi(v,v)}, for every $\varepsilon \in (0,1)$, there holds
\begin{align}\nonumber
\int_{B} |\nabla u - \nabla {\omega}_1|^p dx &\le \varepsilon  \int_{B} \left(\varsigma^p + |\nabla u|^p\right) dx + \max\left\{1, 8\varepsilon^{1 - \frac{2}{p}}\right\} \int_{B} \Psi_{\varsigma}(\nabla u, \nabla {\omega}_1) dx \\ \label{est:DOP-8}
& \le \varepsilon  \int_{B} \left(\varsigma^p + |\nabla u|^p\right) dx + C(p,\varepsilon) \int_{B} \left(\varsigma^{p} + |\mathbf{F}|^p + |\nabla f_2|^p\right) dx,
\end{align}
where the second estimate is an application of~\eqref{est:DOP-7} with suitable value of $\varepsilon_1$. Let ${\omega}_2 \in \mathrm{sol}(\mathbb{A},\mathbb{A}(\cdot,\nabla f_1),{\omega}_1;B)$, since ${\omega}_2 = {\omega}_1 \ge f_1$ almost everywhere on $\partial B$ so it deduces from Lemma~\ref{lem:weak-max} that ${\omega}_2 \ge f_1$ almost everywhere in $B$. Therefore we may take ${\omega}_2$ as the test function in~\eqref{est:DOP-1} to find that
\begin{align*}
\int_{B} \left\langle \mathbb{A}(x,\nabla {\omega}_1), \nabla {\omega}_1 -  \nabla {\omega}_2 \right \rangle dx \le \int_{B} \left \langle \mathbb{A}(x,\nabla f_2), \nabla {\omega}_1 -  \nabla {\omega}_2 \right \rangle dx,
\end{align*}
which with choosing ${\omega}_1-{\omega}_2$ as the test function in variational formula of equation solving ${\omega}_2$, to observe that
\begin{align}\label{est:DOP-9}
\int_{B} \left\langle \mathbb{A}(x,\nabla {\omega}_1) - \mathbb{A}(x,\nabla {\omega}_2), \nabla {\omega}_1 - \nabla {\omega}_2\right\rangle dx = \int_{B}  \left\langle \mathbb{A}(x,\nabla f_2) - \mathbb{A}(x,\nabla f_1), \nabla {\omega}_1 - \nabla {\omega}_2\right\rangle dx.
\end{align}
The proof is essentially the same as the previous one in~\eqref{est:DOP-8}, from~\eqref{est:DOP-9} once again we may show that
\begin{align}\label{est:DOP-10}
\int_{B} |\nabla {\omega}_1 - \nabla {\omega}_2|^p dx \le \varepsilon  \int_{B} \left(\varsigma^p + |\nabla {\omega}_1|^p\right) dx + C(p,\varepsilon) \int_{B} \left(\varsigma^{p} + |\nabla f_1|^p + |\nabla f_2|^p\right) dx,
\end{align}
for every $\varepsilon \in (0,1)$. Let $v \in \mathrm{sol}(\mathbb{A},0,u;B)$ or $v \in \mathrm{sol}(\mathbb{A},0,{\omega}_2;B)$ with notice that ${\omega}_2 = {\omega}_1 = u$ on $\partial B$. The same proof remains valid to obtain the following estimate
\begin{align}\label{est:DOP-11}
\int_{B} |\nabla {\omega}_2 - \nabla v|^p dx \le \varepsilon  \int_{B} \left(\varsigma^p + |\nabla v|^p\right) dx + C(p,\varepsilon) \int_{B} \left(\varsigma^{p} + |\nabla f_1|^p\right) dx,
\end{align}
for all $\varepsilon \in (0,1)$. Collecting the estimates in~\eqref{est:DOP-2}, \eqref{est:DOP-8}, \eqref{est:DOP-10} and~\eqref{est:DOP-11} to discover that
\begin{align*}
\int_{B} |\nabla u - \nabla v|^p dx &\le 3^{p-1}\varepsilon \int_{B} \left(\varsigma^p + |\nabla u|^p + |\nabla v|^p\right)dx \\
& \hspace{2cm} + C(p,\varepsilon) \int_{B} \left(\varsigma^{p} + |\mathbf{F}|^p + |\nabla f_1|^p + |\nabla f_2|^p\right) dx,
\end{align*}
which guarantees~\eqref{eq:lem-Phi(u,v)}, by taking into account the fact that 
$$|\nabla v|^p \le 2^{p-1}(|\nabla u| + |\nabla u - \nabla v|^p),$$ 
and then changing a suitable value of $\varepsilon>0$.
\end{proof}
\subsubsection{Global regularity results}
\begin{theorem}[Global Lorentz estimates under assumption \eqref{hyp:P}]
\label{theo:DOP-1}
Let $\Omega \subset \mathbb{R}^n$ be an open bounded domain satisfying~\eqref{hyp:P} with two constants $c_0, r_0>0$. Assume that operator $\mathbb{A}$ satisfies~\eqref{eq:A1-DOP}-\eqref{eq:A2-DOP} and $\mathbb{B}$ satisfies condition~\eqref{eq:B-DOP}. Suppose that 
\begin{align*}
u \in \mathrm{sol}^{\mathbf{DOP}}\left(\mathbb{A},\mathbb{B}(\cdot,\mathbf{F}),\mathcal{S}_0;\Omega\right),
\end{align*}
with given data $\mathbf{F}$, $f_1$, $f_2$ and $\mathcal{S}_0$ as in Lemma~\ref{lem:u-v-DOP}. Then there exists $\gamma>1$ such that  for $0 \le \alpha < \frac{n}{\gamma}$, $0<q<\frac{n \gamma}{n - \alpha \gamma}$ and $0 < s \le \infty$, there holds
\begin{align}\label{est:theo-DOP-1}
\|\mathbf{M}_{\alpha}((\varsigma +|\nabla u|)^p)\|_{L^{q,s}(\Omega)} \le C\|\mathbf{M}_{\alpha}(\varsigma^p + |\mathbf{F}|^p + |\nabla f_1|^p + |\nabla f_2|^p)\|_{L^{q,s}(\Omega)}.
\end{align}
\end{theorem}
\begin{proof}
Let us consider two functions given by
\begin{align*}
\mathcal{F} = \varsigma^p + |\mathbf{F}|^p + |\nabla f_1|^p + |\nabla f_2|^p \ \mbox{ and }\ \mathcal{G} = (\varsigma +|\nabla u|)^p.
\end{align*}
For every $0< r \le r_0/2$ and $\nu \in \overline{\Omega}$, we now consider
\begin{align*}
v \in \mathrm{sol}(\mathbb{A},0,u;\Omega_{2r}(\nu)),
\end{align*}
and two functions as follows
\begin{align*}
\varphi = (\varsigma + |\nabla v|)^p, \qquad \psi = |\nabla u - \nabla v|^p. 
\end{align*}
A simple computation shows that $(\mathcal{G},\varphi,\psi) \in Q(\Omega_{2r}(\nu))$ and~\cite[Theorem 10]{Mi3} ensures that the existence of $\gamma = \gamma(n,p,\sigma,\Lambda)>1$ such that $\varphi \in \mathcal{RH}^{\gamma}(\Omega_{r}(\nu))$. Thanks to Lemma~\ref{lem:u-v-DOP}, we conclude that $\mathcal{F}$, $\mathcal{G}$ satisfy the local comparison~\ref{ing:A2_1}. Therefore, the proof of~\eqref{est:theo-DOP-1} is straightforward by applying Theorem~\ref{theo:norm-L-1}.
\end{proof}
\begin{remark}
Although not indispensable, the global Lorentz and Orlicz-Lorentz results can also be achieved under assumption~\eqref{hyp:R} and similar to Theorem~\ref{theo:app-2}.
\end{remark}

\begin{remark}
It is possible to improve these results under various additional hypotheses on the nonlinear operator $\mathbb{A}$ and the boundary of the domain $\Omega$. On the other hand, with the best understanding of this approach, we expect that the validity of elliptic problems will be carried on also to the parabolic ones.
\end{remark}


\begin{thebibliography}{99}
\bibitem{AM2007}  E. Acerbi, G. Mingione, {\em Gradient estimates for a class of parabolic systems}, Duke Math. J. {\bf 136} (2007), 285–320.

\bibitem{Adam1975} D.R. Adams, {\em  A note on Riesz potentials}, Duke Math. J. {\bf 42} (1975), 765-778.

\bibitem{Adam1988} D.R. Adams, {\em A note on Choquet integral with respect to Hausdorff capacity}, Lecture Notes in Math., 1302, Springer, Berlin, 1988, pp.115-124.

\bibitem{AH} D.R. Adams,   L.I. Hedberg, {\em Function spaces and potential theory}, Springer-Verlag, Berlin, 1996.

\bibitem{Phuc2015} K. Adimurthi, N.C. Phuc, {\em Global Lorentz and Lorentz-Morrey estimates below the natural exponent for quasilinear equations}, Calc.  Var. Partial Differential Equations  {\bf 54}(3) (2015),  3107--3139.

\bibitem{Baroni2013} P. Baroni, {\em Lorentz estimates for degenerate and singular evolutionary systems}, J. Diff. Equ. {\bf 255} (2013), 2927-2951.

\bibitem{CFL1993} F.  Chiarenza, M. Frasca, P. Longo,  {\em $W^{2,p}$-solvability of the Dirichlet problem for nondivergence elliptic equations with VMO coefficients}, Trans. Amer. Math. Soc. 336 (1993), 841–853.

\bibitem{BCDKS} D. Breit, A. Cianchi, L. Diening, T. Kuusi,  S. Schwarzacher, {\em  The $p$-Laplace system with right-hand side in divergence form: inner and up to the boundary pointwise estimates},  Nonlinear Anal. {\bf 153} (2017),  200-212.

\bibitem{Breit2018} D. Breit, A. Cianchi, L. Diening, T. Kuusi, S. Schwarzacher, {\em Pointwise Calder\'on-Zygmund gradient estimates for the $p$-Laplace system}, J. Math. Pures Appl. {\bf 114} (2018).

\bibitem{BDM2011} V. B\"ogelein, F. Duzzar, G. Mingione, {\em Degenerate problems with irregular obstacles}, J. Reine Angew. Math. {\bf 650} (2011), 107–160.

\bibitem{BS2012} V. B\"ogelein, C. Scheven, {\em Higher integrability in parabolic obstacle problems}, Forum Math., {\bf 24}(5)(2012), 931-972.

\bibitem{BCW2012} S.-S. Byun, Y. Cho, L. Wang, {\em Calder\'on-Zygmund theory for nonlinear elliptic problems with irregular obstacles}, J. Funct. Anal. {\bf 263} (10) (2012) 3117–3143.

\bibitem{ByunKim2016} S.-S. Byun, Y. Kim, {\em Elliptic equations with measurable nonlinearities in nonsmooth domains}, Adv. Math. {\bf 288} (2016), 152–200.

\bibitem{Byun2017JDE} S. S. Byun, J. Oh, {\em Global gradient estimates for the borderline case of double phase problems with BMO coeffcients in nonsmooth domains}. J. Diff. Equ. {\bf 263} (2017), 1643–-1693.

\bibitem{BPR2013} S.-S. Byun, D. K. Palagachev, S. Ryu, {\em Elliptic obstacle problems with measurable coefficients in non-smooth domains}, Numer. Funct. Anal. Optim. {\bf 35} (7–9) (2013), 893–910.

\bibitem{BW1} S.-S. Byun, D. K. Palagachev, P. Shin, {\em Global Sobolev regularity for general elliptic equations of $p$-Laplacian type},  Calc.  Var. Partial Differential Equations  {\bf 57} (2018),  pp 135.

\bibitem{BR2020} S.-S. Byun, S. Ryu, {\em Gradient estimates for nonlinear elliptic double obstacle problems}, Nonlinear Anal. {\bf 194} (2020), 111333.

\bibitem{SSB3} S.-S. Byun, L. Wang, {\em $L^p$-estimates for general nonlinear elliptic equations}, Indiana Univ. Math. J. {\bf 56}(6) (2007), 3193–-3221.

\bibitem{BW2} S.-S. Byun, L. Wang, {\em Elliptic equations with BMO nonlinearity in Reifenberg domains},  Adv. Math. {\bf 219}(6)  (2008), 1937-1971.

\bibitem{SSB1} S.-S. Byun, L. Wang, {\it Nonlinear gradient estimates for elliptic equations of general type},  Calc.  Var. Partial Differential Equations  {\bf 45}(3-4) (2012), 403-419.

\bibitem{SSB4} S.-S. Byun, L. Wang, S. Zhou, {\em Nonlinear elliptic equations with BMO coefficients in Reifenberg domains}, J. Funct. Anal. {\bf 250}(1) (2007), 167–-196.

\bibitem{BYZ2008} S.-S. Byun, F. Yao, S. Zhou, {\em Gradient estimates in Orlicz space for nonlinear elliptic equations}, J. Funct. Anal. {\bf 255}(8) (2008), 1851–-1873.

\bibitem{Caffarelli_obs}  L.A. Caffarelli, {\em The obstacle problem revisited} J. Fourier Anal. Appl. {\bf 4}(4-5)(1998), 383-402.

\bibitem{CC1995} L.A. Caffarelli, X. Cabr\'e, {\it Fully nonlinear elliptic equations},  American Mathematical Society Colloquium Publications, American Mathematical Society, Providence {\bf 43}(1) (1995), 1-21.

\bibitem{CP1998} L.A. Caffarelli, I. Peral, {\it On $W^{1,p}$ estimates for elliptic equations in divergence form},  Commun. Pure Appl. Math. {\bf 51}(1) (1998), 1-21.

\bibitem{Chle2018} I. Chlebicka, {\em Gradient estimates for problems with Orlicz growth},  Nonlinear Anal. (2018), doi:10.1016/j.na.2018.10.008.

\bibitem{Choe1991} H. J. Choe, {\em A regularity theory for a general class of quasilinear elliptic partial differential equations and obstacle problems}, Arch. Rational Mech. Anal. 114, 383–394 (1991).

\bibitem{Choe2016} H. J. Choe,  P. Souksomvang, {\em Elliptic gradient constraint problem}, Comm. in Partial Differential Equations, {\bf 41}(12), 1918-1933.

\bibitem{CoMi2016} M. Colombo, G. Mingione, {\it Calder\'on-Zygmund estimates ans non-uniformly elliptic operators},  J. Funct. Anal. {\bf 136}(4) (2016), 1416-1478.

\bibitem{MMV1989} G. Dal Maso, U. Mosco, M.A. Vivaldi, {\em A pointwise regularity theory for the two-obstacle problem}, Acta Math. {\bf 163} (1–2)(1989), 57–107.

\bibitem{DiBenedetto1983} E. DiBenedetto, {\em $C^{1+\alpha}$ local regularity of weak solutions of degenerate elliptic equations},  Nonlinear Anal. {\bf 7}(8) (1983), 827-850. 

\bibitem{DiBenedettobook} E. DiBenedetto, {\em Degenerate parabolic equations}, Universitext. Springer-Verlag, New York, 1993.

\bibitem{DiBenedetto1993} E. DiBenedetto, J. Manfredi, {\em On the higher integrability of the gradient of weak solutions of certain degenerate elliptic systems},  Amer. J. Math. {\bf 115}(5) (1993),  1107-1134.

\bibitem{Duzamin2} F. Duzaar and G. Mingione, {\em Gradient estimates via linear and nonlinear potentials}, J. Funt. Anal. {\bf 259} (2010), 2961-2998.

\bibitem{55DuzaMing} F. Duzaar and G. Mingione, {\em Gradient estimates via non-linear potentials},  Amer. J. Math. {\bf 133} (2011), 1093-1149. 

\bibitem{Eleuteri2007} M. Eleuteri, {\em Regularity results for a class of obstacle problems}, Appl Math {\bf 52}, 137–170 (2007).

\bibitem{EH2008} M. Eleuteri, J. Habermann, {\em Calder\'on-Zygmund type estimates for a class of obstacle problems with $p()x$ growth}, J. Math. Anal. Appl. {\bf 344}(2)(2008), 1120-1142.

\bibitem{EH2011} M. Eleuteri, J. Habermann, {\em A H\"older continuity result for a class of obstacle problems under non standard growth conditions}, Math. Nachr. 284 (11–12) (2011), 1404–1434.

\bibitem{Evans1982} L. Evans, {\em A new proof of local $C^{1,\alpha}$ regularity for solutions of certain degenerate elliptic PDE}, J. Differential Equations {\bf 145} (1982), 356-373.

\bibitem{FT2018} F.D. Fazio, T. Nguyen, {\em Regularity estimates in weighted Morrey spaces for quasilinear elliptic equations}, Rev. Mat. Iberoamericana, (Doi: 10.4171/rmi/1178).

\bibitem{Friedman} A. Friedman, {\em Variational Principles and Free-Boundary Problems}, Wiley–Interscience Publication, Pure Appl. Math., John Wiley and Sons, Inc., New York, 1982.

\bibitem{Gehring} F. W. Gehring, {\em The $L^p$-integrability of the partial derivatives of a quasiconformal mapping}, Acta Math. {\bf 130} (1973), 265–277.

\bibitem{Giu} E. Giusti, {\em Direct methods in the calculus of variations}, World Scientic Publishing Co., Inc., River Edge, NJ, 2003.

\bibitem{Gra97} L. Grafakos,  S. Montgomery-Smith, {\em Best constants for uncentred maximal functions},  Bull. London Math. Soc. {\bf 29} (1997) 60–-64.

\bibitem{55Gra} L. Grafakos, {\em Classical and Modern Fourier Analysis}, Pearson/Prentice Hall, 2004.

\bibitem{Hasto} P. Harjulehto, P. H{\"a}st{\"o}, {\em Orlicz spaces and Generalized Orlicz spaces}, reprint 2018.

\bibitem{Iwaniec83} T. Iwaniec, {\em Projections onto gradient fields and  $L^p$-estimates for degenerated elliptic operators}, Stud. Math. {\bf 75}(3) (1983), 293-–312.

\bibitem{Iwaniec1995} T. Iwaniec, {\em The Gehring lemma. Quasiconformal mappings and analysis} (Ann Arbor, MI, 1995), 181–204, Springer, New York, 1998.

\bibitem{IS1994} T. Iwaniec, C. Sbordone, {\em Weak minima of variational integrals}, J. Reine Angew. Math. (Crelles J.) {\bf 454} (1994), 143-161.

\bibitem{Kaminska} A. Kaminska, {\em Some Remarks on Orlicz-Lorentz Spaces}, Mathematische Nachrichten {\bf 147}(1)(1990), 29 - 38.

\bibitem{KZ1991} T. Kilpel\"ainen, W.P. Ziemer, {\em Pointwise regularity of solutions to nonlinear double obstacle problems}, Ark. Mat. {\bf 29} (1)(1991), 83–106.

\bibitem{KS1980} D. Kinderlehrer, G. Stampacchia, {\em An Introduction to Variational Inequalities and Their Applications}, Pure Appl. Math., vol. 88, Academic Press, New York, London, 1980.

\bibitem{K1997} J. Kinnunen, {\em The Hardy-Littlewood maximal function of a Sobolev function}, Israel J. Math. {\bf 100} (1997), 117-–124.

\bibitem{KS2003} J. Kinnunen, E. Saksman, {\em Regularity of the fractional maximal function}, Bull. London Math. Soc. {\bf 35} (2003), no. 4, 529–535.

\bibitem{KZ} J. Kinnunen, S. Zhou, {\em A local estimate for nonlinear equations with discontinuous coefficients}, Commun. Partial Differ. Equ. {\bf 24}(11-12) (1999) 2043-2068.

\bibitem{KS1980} N. V. Krylov, M. V. Safonov, {\em A property of the solutions of parabolic equations with measurable coefficients}, Izv. Akad. Nauk SSSR Ser. Mat. {\bf 44} (1980), no. 1, 161–175, 239. MR 563790.

\bibitem{Krylov2007} N. V. Krylov, {\em Parabolic and Elliptic Equations with VMO Coeffcients}, Comm. Part. Diff. Equ., {\bf 32} (2007), 453-475.

\bibitem{KM2012} T. Kuusi, G. Mingione, {\em Universal potential estimates}. J. Funct. Anal. {\bf 262}(10) (2012), 4205-–4269.

\bibitem{KM2014} T. Kuusi, G. Mingione, {\em Guide to nonlinear potential estimates}, Bull. Math. Sci. {\bf 4}(1) (2014), 1-82.

\bibitem{Lewis1983} J.L. Lewis, {\em Regularity of the derivatives of solutions to certain degenerate elliptic equations}, Indiana Univ. Math. J. {\bf 32} (1983), 849-858.

\bibitem{Lewis93} J.L. Lewis, {\em On very weak solutions of certain elliptic systems}, Comm. Partial Differential Equations {\bf 18} (1993), 1515-1537.

\bibitem{Lieberman1988} G.M. Lieberman, {\em Boundary regularity for solutions of degenerate elliptic equations},  Nonlinear Analysis {\bf 12}(11) (1988), 1203-1219.

\bibitem{Lieberman1991} G.M. Lieberman, {\em Regularity of solutions to some degenerate double obstacle problems}, Indiana Univ. Math. J. {\bf 40} (3)(1991), 1009–1028.

\bibitem{MP11} T. Mengesha, N. C. Phuc, {\em Weighted and regularity estimates for nonlinear equations on Reifenberg flat domains},   J. Diff. Equ. {\bf 250} (2011), 1485-2507.

\bibitem{MP12} T. Mengesha, N.C. Phuc, {\it Global estimates for quasilinear elliptic equations on Reifenberg flat domains},  Arch. Ration. Mech. Anal. {\bf 203}(1)  (2012), 189-216.

\bibitem{MZ1986} J.H. Michael, W. P. Ziemer, {\em Interior regularity for solutions to obstacle problems}, Nonlinear Anal., {\bf 10}(12)(1986), 1427-1448.

\bibitem{Mi3} G. Mingione, {\em Gradient estimates below the duality exponent}, Math. Ann. {\bf 346} (2010), 571-627.

\bibitem{MontOL} S. Montgomery-Smith, {\em Comparison of Orlicz-Lorentz spaces}, Stud. Math., {\bf 103}(2)(1992), 161-189.

\bibitem{MW} B. Muckenhoupt, R. L. Wheeden, {\em Weighted norm inequalities for fractional integrals}, Trans. Amer. Math. Soc. {\bf 192} (1974), 261–274.

\bibitem{55QH4} Q.-H. Nguyen, N.C. Phuc, {\em  Good-$\lambda$ and Muckenhoupt-Wheeden type bounds, with applications to quasilinear elliptic equations with gradient power source terms and measure data}, Math. Ann. {\bf 374}(1-2) (2019), 67-98.

\bibitem{PNmix} T.-N. Nguyen, M.-P. Tran, {\em Lorentz improving estimates for the $p$-Laplace equations with mixed data}, arXiv:2003.04530.

\bibitem{Orlicz1932} W. Orlicz, {\em \"Uber eine gewisse Klasse von R\"aumen vom Typ B}, Bull. Intern. Pol., {\bf 8}(1932), 207-220.

\bibitem{Tuoc2018} T. Phan, {\em Regularity estimates for BMO-weak solutions of quasilinear elliptic equations with inhomogeneous boundary conditions}, T. Nonlinear Differ. Equ. Appl. {\bf 25}(8) (2018), https://doi.org/10.1007/s00030-018-0501-2.

\bibitem{Phuc2} N. C. Phuc, {\em Morrey global bounds and quasilinear Riccati type equations below the natural exponent}, J. Math. Pures et Appliquées. {\bf 102} (2014), 99-123.

\bibitem{Rao1991} M.M. Rao, Z. D. Ren, {\em Theory of Orlicz spaces}, volume {\bf 46} of Monographs and Textbooks in Pure and Applied Mathematics, Marcel Dekker, Inc., New York, 1991.

\bibitem{Rodfrigues1987} J.F. Rodrigues, {\em Obstacle Problems in Mathematical Physics}, North Holland, Amsterdam (1987).

\bibitem{RT2011} J.F. Rodrigues, R. Teymurazyan, {\em On the two obstacles problem in Orlicz-Sobolev spaces and applications}, Complex Var. Elliptic Equ. {\bf 56} (7–9) (2011), 769–787.

\bibitem{Stein} E. Stein, {\em Singular Integrals and Differentiability Properties of Functions}, Princton Univ. Press, Princeton, New Jersey, 1970.

\bibitem{MPT2018} M.-P. Tran, {\em Good-$\lambda$ type bounds of quasilinear elliptic equations for the singular case},  Nonlinear Analysis {\bf 178} (2019), 266-281.

\bibitem{PNCRM} M.-P. Tran, T.-N. Nguyen, {\em Generalized good-$\lambda$ techniques and applications to weighted Lorentz regularity for quasilinear elliptic equations}, Comptes Rendus Mathematique {\bf 357}(8) (2019), 664-670.

\bibitem{PNJDE} M.-P. Tran, T.-N. Nguyen, {\em New gradient estimates for solutions to quasilinear divergence form elliptic equations with general Dirichlet boundary data}, J. Diff. Equ. {\bf 268}(4) (2020), 1427-1462.

\bibitem{PNCCM} M.-P. Tran, T.-N. Nguyen, {\em Lorentz-Morrey global bounds for singular quasilinear elliptic equations with measure data}, Communications in Contemporary Mathematics (2019). https://doi.org/10.1142/S0219199719500330.

\bibitem{MPTNsub} M.-P. Tran, T.-N. Nguyen, {\em Weighted Lorentz gradient and point-wise estimates for solutions to quasilinear divergence form elliptic equations with an application}, arXiv:1907.01434.

\bibitem{PNnonuniform} M.-P. Tran, T.-N. Nguyen, {\em Global Lorentz estimates for non-uniformly nonlinear elliptic equations via fractional maximal operators}, Journal of Mathematical Analysis and Applications. DOI:10.1016/j.jmaa.2020.124084.

\bibitem{Troianiello} G.M. Troianiello, {\em Elliptic Differential Equations and Obstacle Problems}, The University Series in Mathematics. Plenum Press, New York, xiv+353 pp. ISBN: 0-306-42448-7, (1987).

\bibitem{Tolksdorf1984} P. Tolksdorff, {\em Regularity for a more general class of quasilinear elliptic equations},  J. Differential Equations  {\bf 51}(1) (1984),   126-150. 

\bibitem{Uhlenbeck1977} K. Uhlenbeck, {\em Regularity for a class of nonlinear elliptic systems}, Acta. Math. {\bf 138}(3-4) (1977), 219-240.

\bibitem{Ural1968} N. Ural\'tzeva, {\em Degenerate quasilinear elliptic systems (Russian)}, Zap. Naucm. Sem. Leningrad. Otdel. Mat. Inst. Steklov. (LOMI) {\bf 7} (1968), 184-222.

\bibitem{Vitali08} G. Vitali, {\it Sui gruppi di punti e sulle funzioni di variabili reali}, Atti Accad. Sci. Torino {\bf 43} (1908), 229--246.

\bibitem{Wang2} L. Wang, {\em A geometric approach to the Calder\'on-Zygmund estimates}, Acta Math. Sin.(Engl. Ser.) {\bf 19} (2003), 381–-396.
\end{thebibliography}
\end{document}